\documentclass[10pt]{amsart}
\usepackage{amssymb,amsmath,textcomp,amsfonts,amsthm}
\usepackage[utf8]{inputenc}
\usepackage{url}
\usepackage{mathtools}
\usepackage{romannum}
\usepackage{color}
\usepackage[normalem]{ulem}
\usepackage{tikz}
\newtheoremstyle{normal}
{10pt}
{10pt}
{\normalfont}
{}
{\bfseries}
{}
{0.8em}
{\bfseries{\thmname{#1}\thmnumber{ #2}.\thmnote{ \hspace{0.5em}(#3)\newline}}}
 
\newtheoremstyle{kursiv}
{10pt}
{10pt}
{\itshape}
{}
{\bfseries}
{}
{0.8em}
{\bfseries{\thmname{#1}\thmnumber{ #2}.\thmnote{ \hspace{0.5em}(#3)\newline}}}

\theoremstyle{kursiv}
\newtheorem{definition}{Definition}[section]
\newtheorem{satz}[definition]{Theorem}
\newtheorem{lemma}[definition]{Lemma}
\newtheorem{korollar}[definition]{Corollary}
\newtheorem{proposition}[definition]{Proposition}

\theoremstyle{normal}

\newtheorem{bemerkung}[definition]{Remark}
\newtheorem{beispiel}[definition]{Example}

\newlength{\leftstackrelawd}
\newlength{\leftstackrelbwd}
\def\leftstackrel#1#2{\settowidth{\leftstackrelawd}%
{${{}^{#1}}$}\settowidth{\leftstackrelbwd}{$#2$}%
\addtolength{\leftstackrelawd}{-\leftstackrelbwd}%
\leavevmode\ifthenelse{\lengthtest{\leftstackrelawd>0pt}}%
{\kern-.5\leftstackrelawd}{}\mathrel{\mathop{#2}\limits^{#1}}}

\newcommand{\iR}{\mathbb{R}}
\newcommand{\R}{\mathbb{R}}

\newcommand{\N}{\mathbb{N}}
\newcommand{\iZ}{\mathbb{Z}}
\newcommand{\Z}{\mathbb{Z}}

\allowdisplaybreaks
\begin{document}
\pagenumbering{arabic}
\title{Aronson-B\'{e}nilan and Harnack estimates for the discrete porous medium equation}
\author{Sebastian Kr\"ass}
\email{sebastian.kraess@uni-ulm.de}
\author{Rico Zacher$^*$}
\thanks{$^*$Corresponding author}
\email[Corresponding author:]{rico.zacher@uni-ulm.de}
\address[Sebastian Kr\"ass, Rico Zacher]{Institute of Applied Analysis, Ulm University, Helmholtzstra\ss{}e 18, 89081 Ulm, Germany.}
\thanks{Sebastian Kr\"ass is supported by a PhD-scholarship of the ``Hanns-Seidel-Stiftung'', Germany. Rico Zacher was supported by the DFG (project number 355354916, GZ ZA 547/4-2).}
\thanks{{\bf Declarations:} On behalf of all authors, the corresponding author states that there is no conflict of interest. Data availability
statement: not applicable.}

\begin{abstract}
We consider the porous medium equation (PME) on a locally finite graph and identify suitable curvature-dimension (CD) conditions under which
a discrete version of the fundamental Aronson-B\'enilan estimate holds true for positive solutions of the PME. We also show
that these estimates allow to prove Harnack inequalities which are structurally similar to the continuous case. The new CD conditions
are illustrated with several concrete examples, e.g.\ complete and chain-like graphs. 
\end{abstract}
\maketitle

${}$


\section{Introduction}
In recent years, discrete versions of the celebrated Li-Yau inequality (\cite{Li,LY})
\begin{equation}\label{originalLiYau}
- \Delta (\log u) \leq \frac{d}{2t},\quad t>0,
\end{equation}
for positive solutions $u$ of the heat equation $\partial_t u = \Delta u$, $t>0$, on complete Riemannian manifolds
with nonnegative Ricci curvature and topological dimension $d$ ($\Delta$ denoting the Laplace-Beltrami operator), have been established under suitable curvature-dimension conditions on the underlying discrete structure (\cite{BHL,DKZ,MUN,MUN2}).
The analogue of the Li-Yau inequality in case of the  
porous medium equation (PME)
\begin{equation}
\label{MainEqu}
\partial_t u - \Delta u^m = 0, \quad t>0,
\end{equation}
with $m>1$ is the Aronson-B\'enilan estimate
\begin{equation}
\label{AronsBenOrig}
-\Delta \Big( \frac{m}{m-1}u^{m-1}\Big) \le  \frac{\kappa}{t},\quad t>0,
\end{equation}
where $\kappa = \frac{d}{2+d(m-1)}$. Observe that sending $m\to 1$ in \eqref{AronsBenOrig} leads to \eqref{originalLiYau}.
In 1979, Aronson and B\'enilan proved \eqref{AronsBenOrig} for positive smooth
solutions of \eqref{MainEqu} in $\iR^d$ (\cite{AB}). Corresponding results on manifolds were later obtained in \cite{LNV}, see also
\cite{PME} and \cite{CZ}. Estimate \eqref{AronsBenOrig} has been a key step in the development of the theory for the PME in $\iR^d$,
see the monograph by V\'azquez \cite{PME}.

One of the main objectives of this paper is to derive Aronson-B\'enilan type estimates for the PME in a {\em discrete} setting. 
Concerning existence and uniqueness as well as qualitative properties of solutions to the discrete PME in a rather general framework we refer to the recent work \cite{BSW}.

Let $X$ be a finite or 
countably infinite set. For functions $u:X \rightarrow \iR$, we consider the operator
\begin{equation}\label{eq:operator}
L u(x) = \sum\limits_{y \in X} k(x,y) \big( u(y) - u(x)\big), \quad x \in X,
\end{equation}
where $k: X \times X \to [0,\infty)$ is a (non-trivial) kernel such that for any $x\in X$ there are at most finitely many $y\in X$
with $k(x,y)>0$. The kernel $k$ naturally induces a locally finite, directed graph $(V,E)$ where $V = X$ is the set of vertices and $(x,y) \in E$ (the set of edges) if and only if $k(x,y) > 0$. The weight of an edge $(x,y)$ is then given by $k(x,y)$. 
The operator $L$ can be regarded as a generalized Laplacian on the graph $(V,E)$. Even though a larger part of our general results are valid for weighted and directed graphs, in our 
examples we focus on unweighted graph structures, i.e.\ $k$ is symmetric and only takes values in $\{0,1\}$. 
Clearly, the value of $k(x,x)$ for $x\in X$ does not play a role in \eqref{eq:operator}. We set $k(x,x)=0$, $x\in X$.
However, $L$ can also be interpreted as the generator of a continuous-time Markov chain with state space $X$. The transition rates are then given by $k(x,y)$ if $x \neq y$ and $k(x,x)=-\sum_{y\in X\setminus\{x\}}k(x,y)$ for any $x \in X$.

We consider positive solutions $u: (0,\infty) \times X \to (0,\infty)$ of the discrete porous medium equation
\begin{equation}
\partial_t u (t,x) - L u^m (t,x) = 0, \quad t>0,\,x\in X,
\label{PME}
\end{equation}
where $m>1$, and aim at proving discrete versions of \eqref{AronsBenOrig}. As in the continuous case, the 
{\em pressure} (cf.\ \cite{PME})
\begin{equation}
v = \frac{m}{m-1} u^{m-1}
\label{DefV}
\end{equation}
plays a key role in Aronson-B\'enilan type estimates. Note that \eqref{AronsBenOrig} just means that $-\Delta v\le \frac{\kappa}{t}$. If $u$ is a positive solution of the PME on $\iR^d$, then 
\begin{equation}
\partial_t v = (m-1) v \Delta v + \vert\nabla v \vert^2.
\label{ChainRuleV}
\end{equation}
This evolution equation for the pressure is the starting point in deriving \eqref{AronsBenOrig} in the classical case. The basic
idea is to apply $\Delta$ to \eqref{ChainRuleV}, to use Bochner's formula (with $|\cdot|_{HS}$ denoting the Hilbert-Schmidt norm)
\[
\Delta \vert\nabla v \vert^2=2\vert \nabla^2 v\vert_{HS}^2+2\nabla v\cdot \nabla\Delta v 
\]
and the inequality 
\begin{equation} \label{CDclass}
\vert \nabla^2 v\vert_{HS}^2\ge \frac{1}{d}(\Delta v)^2
\end{equation}
to show that $\Delta v$ is a supersolution
of a certain (degenerate) parabolic equation with a quadratic term $c (\Delta v)^2$ (with $c>0$) on the right-hand side,
for which $-\frac{\kappa}{t}$ is a solution. Applying {\em formally} the comparison principle yields the desired inequality
$-\Delta v\le \frac{\kappa}{t}$.

If one wants to follow the same strategy in the discrete case one is confronted with the difficulty that the chain rule (valid
for differential operators) fails to hold. This concerns both the derivation of the equation for $v$ and the second step where
the Laplacian is applied to this equation. Another problem is to find a suitable replacement for \eqref{CDclass}, which
is a simple example of a curvature-dimension (CD) inequality in the sense of Bakry and \'Emery, the condition $CD(0,d)$ with $d\in [1,\infty)$. The latter is based on
the $\Gamma$-calculus involving the carr\'e du champ operator $\Gamma$ and its iterate $\Gamma_2$ for generators
of Markov semigroups, see \cite{BBL,BGL,BL}, which also contain results on Li-Yau inequalities for Markov diffusion operators satisfying $CD(0,d)$.

For the discrete heat equation, i.e.\ \eqref{PME} with $m=1$, these difficulties could be overcome in 
\cite{BHL,DKZ,MUN,MUN2} by introducing various suitable new $CD$-conditions and identifying certain discrete identities which allow to bypass the lack of chain rule, see also \cite{We,WZ} for $CD$-conditions encoding positive curvature. In \cite{BHL}, the authors work with the square root of $u$ instead of $\log u$
and introduce an exponential $CD$-inequality. A more general calculus, which also leads to logarithmic estimates, is
studied in \cite{MUN,MUN2}. The most precise (and in some cases even optimal) Li-Yau type estimates for $\log u$ are obtained in \cite{DKZ}, where one of the key new ideas is to replace the square $(Lv)^2$ appearing in the classical $CD$-inequality
(and also used in \cite{BHL, MUN}) by a more general term $F(-Lv)$ involving a so-called $CD$-function $F$. This allows for more flexibility and other
relaxation functions than those of the form $c/t$ (as in \eqref{originalLiYau}), and it opens up a corresponding theory for
discrete jump operators with arbitrary long jumps like the fractional discrete Laplacian (\cite{KWZ}),
for which the classical Bakry-\'Emery condition $CD(0,d)$ is violated for all finite $d>0$ (\cite{SWZ1}).

In this paper, we follow the approach from \cite{DKZ}. The first step consists in determining the equation for the pressure. 
Let us recall that in \cite{DKZ}, it was shown that if $u$ is a positive solution of the discrete heat equation, then
\begin{equation} \label{logequ}
\partial_t \log u-L \log u=\Psi_\Upsilon(\log u),
\end{equation}
where $\Psi_\Upsilon (f) (x) = \sum_{y \in X} k(x,y) \Upsilon \big(f(y)-f(x)\big)$ with the function $\Upsilon (r) = e^r-1-r$.
Comparing \eqref{logequ} with the continuous case, where $\partial_t \log u-\Delta \log u=|\nabla \log u|^2$, one sees that 
$\Psi_\Upsilon(w)$ is here the natural discrete replacement for $|\nabla w|^2$ where $w=\log u$. For the discrete PME with
$m>1$,
we identify an operator $\tilde{\Psi}_\Upsilon^{(m)}$ (see Definition \ref{DefTildePsi}) with which the discrete substitute of the continuous pressure equation \eqref{ChainRuleV} takes the form
\begin{equation} \label{pressureequdiscrete}
\partial_t v = (m-1) v L v + \tilde{\Psi}_\Upsilon^{(m)} (v).
\end{equation}
Thus $\tilde{\Psi}_\Upsilon^{(m)} (v)$ replaces $|\nabla v|^2$, a phenomenon which also occurs in the dissipation term
of the R\'{e}nyi entropy, see Proposition \ref{basicprop} below. At this point we remark that Erbar and Maas were able to show that for $m\in (0,2]$ the gradient flow equation
of the R\'{e}nyi entropy with respect to a suitable non-local transportation metric is precisely given by the discrete PME (\cite{ERM}), thereby establishing a discrete counterpart to Otto's seminal contribution \cite{Otto}.

Next, we look at the evolution equation for $Lv$, which is of the form $\partial_t Lv=\mathcal{D}_{m} (u)$. 
Proceeding as in \cite{DKZ}, we introduce a new $CD$-condition, $CD_m (0,d)$ (cf.\ Definition \ref{DefCD}), which 
for any $x\in X$ allows
to estimate $\mathcal{D}_{m} (u)(x)$ from below by $\frac{1}{d}(Lv(x))^2$ for all positive functions $u$ for which $-Lv$
($v$ denoting the corresponding pressure) has a positive local maximum at $x\in X$. We then prove that for any
finite graph satisfying $CD_m (0,d)$ with some $d>0$ the discrete Aronson-Bénilan type estimate
\begin{align} \label{discreteAB}
-L v (t,x) \leq \frac{d}{t},\quad t>0,\,x\in X,
\end{align}
holds true for any positive solution $u$ to the PME on $(0,\infty) \times X$. We also present several positive examples of 
finite graphs that satisfy $CD_m (0,d)$, e.g.\ unweighted complete graphs with an arbitrary number of vertices $D \in \N$. However, there are also negative examples such as chain-like graphs (in particular the very important example of the discrete Laplacian on $\mathbb{Z}$) which fail to satisfy the condition $CD_m (0,d)$ for any $d>0$ at least if $m \geq 2$ and the chain contains more than 4 vertices. Here the case $m\in (1,2)$ remains an interesting open problem. 
Motivated by the failure of $CD_m (0,d)$ for the discrete Laplacian on $\iZ$ in the case $m\ge 2$, we further introduce a more general curvature-dimension condition, $CD_{m,\alpha} (0,d)$ with an additional parameter $\alpha \in [0,1]$ (cf.\ Definition \ref{DefCDGen}) and prove that the discrete Laplacian on $\mathbb{Z}$ satisfies $CD_{m,1} \big(0,\frac{1}{m-1}\big)$
for any $m>1$.

It should be pointed out that, in contrast to the (discrete) linear case $m=1$, where other than square functions may be advantageous as $CD$-functions, in the PME case the square is the natural choice as this leads to a relaxation function
of the form $c/t$. This is the natural form in view of the scaling property that $\lambda u(\lambda^{m-1}t,x)$ solves
the PME for any $\lambda>0$ provided $u$ is a positive solution of the PME. In fact, assuming an estimate of the form
$-\Delta v\le C/t^\gamma$ this scaling property implies that $\gamma=1$.

Inserting the Aronson-B\'enilan estimate \eqref{discreteAB} into the equation for the pressure \eqref{pressureequdiscrete} yields the 
{\em differential} Harnack type inequality
\begin{equation} \label{DiffHar}
\partial_t v \ge \tilde{\Psi}_\Upsilon^{(m)} (v)-\frac{(m-1)d v}{t},\quad t>0,\,x\in X,
\end{equation}
which is the discrete version of the corresponding inequality
\[
\partial_t v \ge |\nabla v|^2-\frac{(m-1)d v}{t}
\]
from the continuous case. We prove that from \eqref{DiffHar} one can deduce a Harnack inequality structurally
similar to the continuous case, which has been studied by Auchmuty and Bao in \cite{AUB}. 
In fact, assuming that the underlying (finite) graph is connected and undirected, and denoting by $k_{\min}$ the minimal positive
weight of an edge, \eqref{DiffHar} implies that
\begin{equation}
t_1^\mu v(t_1,x_1) \leq t_2^\mu v(t_2,x_2) + \frac{2 d(x_1,x_2)^2 \left(t_2^{\mu+1}-t_1^{\mu+1} \right)}
{(\mu+1) k_{\min} \left(t_2-t_1 \right)^2},\quad 0<t_1<t_2,\,x_1,x_2\in X,
\label{HarnackIntro}
\end{equation} 
where $\mu=d(m-1)$ and $d(x_1,x_2)$ is the distance of $x_1$ and $x_2$ defined as the minimal length of all connecting paths. The proof is partly inspired by \cite{AUB},
but also uses discrete techniques that generalize certain estimates from \cite{DKZ} from the case $m=1$. Note that
a Harnack inequality can also be inferred from an Aronson-B\'enilan estimate resulting from the $CD$-condition
$CD_{m,\alpha} (0,d)$ with $d>0$ and $\alpha \in [0,1)$. In the case $\alpha=1$ one only gets the inequality
$\partial_t v\ge -\frac{d(m-1)v}{t}$, which is too weak to deduce a Harnack inequality as the 
'gradient term' $\tilde{\Psi}_\Upsilon^{(m)} (v)$ is missing. This concerns in particular the discrete Laplacian on $\iZ$, for which 
the question, whether an Aronson-B\'enilan estimate similar to the continuous case holds, remains an interesting open problem.

The article is organised as follows. In Section 2, we define the operator $\tilde{\Psi}_\Upsilon^{(m)}$ and derive the evolution equation for the pressure $v$. In Section 3, we introduce the curvature-dimension condition $CD_m (0,d)$ and provide some important positive examples. Section 4 is devoted to chain-like graphs and also includes a discussion of the discrete Laplacian on $\Z$. Moreover, the more general condition $CD_{m,\alpha} (0,d)$ is introduced and shown to be valid for the discrete Laplacian with $\alpha=1$ and $d=\frac{1}{m-1}$. In Section 5, we derive Aronson-Bénilan estimates for finite graphs
satisfying $CD_{m,\alpha} (0,d)$. Finally, in Section 6, Harnack inequalities are deduced from the Aronson-Bénilan estimates

\section{Evolution equation for the pressure $v$}
Let $m>1$, $X$ be a countable set (e.g.\ a finite set) and $k: X \times X \to [0,\infty)$ be nontrivial. For simplicity and with regard to our examples we assume that for any $x\in X$ there are at most finitely many points $y\in X$ with $k(x,y)>0$. Let $L$ be the operator associated with $k$ given by \eqref{eq:operator}. Given $H:\iR\to \iR$ we
further define for functions $w: X\to \iR$ and $x\in X$
\[
\Psi_H(w)(x)=\sum_{y\in X}k(x,y)H\big(w(y)-w(x)\big),
\]
cf.\ \cite[Section 2]{DKZ}.

The following definition introduces an operator which can be viewed in several respects as a discrete replacement for the operator $\vert \nabla v \vert^2$ in the context of the discrete PME.
 
\begin{definition}
\label{DefTildePsi} Let $m>1$ and the function $\tilde{\Upsilon}: \R \to \R$ be given by 
\[
\tilde{\Upsilon} (r) = \frac{(m-1)^2}{m} \Upsilon \big(\frac{m}{m-1} r \big) - (m-1) \Upsilon (r),
\]
with $\Upsilon(r)=e^r-1-r$, $r\in \iR$.
For functions $u: X\to (0,\infty)$ we define
\begin{align*}
\tilde{\Psi}_{\Upsilon}^{(m)}& (u)(x) = u^2(x)\Psi_{\tilde{\Upsilon}} (\log u)(x) \\&=\sum_{y \in X} k(x,y) \Big(\frac{m-1}{m} u^2(x) + \frac{(m-1)^2}{m} u^\frac{m-2}{m-1} (x) u^\frac{m}{m-1} (y) - (m-1) u(x) u(y)\Big). 
\end{align*}
\end{definition}

The importance of the operator $\tilde{\Psi}_{\Upsilon}^{(m)}$ becomes clear with the following proposition, which
states that this operator replaces $\vert \nabla v \vert^2$ from the continuous setting in two significant places:
(i) in the evolution equation for the pressure and (ii) in the dissipation term of the R\'{e}nyi entropy.
\begin{proposition} \label{basicprop}
Let $m>1$ and $u:(0,\infty)\times X\rightarrow (0,\infty)$ be a solution to $\partial_t u - Lu^m = 0$ on $(0,\infty) \times X$. Then the following holds true.
\begin{enumerate}
\item[(i)] The pressure function $v = \frac{m}{m-1} u^{m-1}$ solves the evolution equation
\begin{equation}
\partial_t v = (m-1) vLv + \tilde{\Psi}_\Upsilon^{(m)} (v)\quad \mbox{on}\;(0,\infty)\times X.
\label{DiscreteChainRuleV}
\end{equation}
\item[(ii)] Suppose in addition that $u(t,\cdot) \in \ell^{m} (X)$ for all $t>0$ and let $\mu$ denote an invariant and reversible finite measure on $X$ w.r.t.\ $L$, with density $\pi$ w.r.t.\ to the counting measure, that is in particular the detailed
balance condition
\[
k(x,y)\pi(x)=k(y,x)\pi(y),\quad x,y\in X,
\]
is satisfied. Then 
\begin{align}
\frac{d}{dt} \int_X \frac{u^m}{m(m-1)} d\mu = -\frac{1}{m} \int_X u \tilde{\Psi}_\Upsilon^{(m)} (v) d\mu,
\quad t>0.
\label{DiscreteEntropyEqu}
\end{align}
\end{enumerate}
\end{proposition}
\begin{proof}
(i) From \cite[Example 2.4]{DKZ} we know that
\begin{equation*}
L (u^m) = m u^m L \log (u) + u^m \Psi_{\Upsilon_m} (\log u) = u^m \Psi_{\Upsilon^\prime} \big(m \log(u)\big),
\end{equation*}
where $\Upsilon_\beta(r)=\Upsilon(\beta r)$, $r\in \iR$, $\beta>0$.
This yields
\begin{align*}
\partial_t v &= m u^{m-2} \partial_t u = m u^{m-2} L u^m \notag = m^2 u^{2m-2} \Big(L \log u + \frac{1}{m} \Psi_{\Upsilon_m} \big(\log u \big) \Big) \notag \\ &= m^2 u^{2m-2} \Big(\frac{1}{m-1} L \log v + \frac{1}{m} \Psi_{\Upsilon_{\frac{m}{m-1}}} (\log v) \Big) \notag \\ &= (m-1)^2 v^2 \Big(\frac{1}{m-1} \Big(\frac{L v}{v} - \Psi_\Upsilon (\log v) \Big)+\frac{1}{m} \Psi_{\Upsilon_{\frac{m}{m-1}}} (\log v) \Big) \notag \\ &= (m-1) v Lv + \tilde{\Psi}_\Upsilon^{(m)} (v). 
\end{align*}
(ii) 
Note that $\int_X f d\mu=\sum_{x\in X} f(x) \pi(x)$ for any $f\in L^1(X,\mu)$. We have 
\begin{align*}
\frac{d}{dt} \int_X \frac{u^m}{m(m-1)} d\mu &= \frac{1}{m-1} \int_X u^{m-1} L u^m d\mu = \frac{1}{m-1} \int_X u^m L u^{m-1} d \mu,
\end{align*}
by reversibility of $\mu$. Next, applying \cite[Lemma 2.1]{DKZ} to the function $H(r) = m r^\frac{m-1}{m}$ we see
that
\begin{align*}
L H(u^m)(x) = H^\prime \big(u^m(x)\big) Lu^m (x) + \sum_{y \in X} k(x,y) \Lambda_H \big(u^m(y),u^m(x)\big)
\end{align*}
for any $x \in X$, where $\Lambda_H (w,z) = H(w)-H(z)-H^\prime (z)(w-z), \ w,z \in \R$. This is equivalent to 
\begin{align*}
L (mu^{m-1})(x) = (m-1) \frac{Lu^m (x)}{u(x)}& + \sum_{y \in X} k(x,y) \Big(mu^{m-1}(y)-mu^{m-1}(x)\\ &-(m-1)\big(\frac{u^m(y)}{u(x)}-u^{m-1}(x)\big)\Big).
\end{align*}
Thus,
\begin{align*}
\frac{d}{dt} &\int_X \frac{u^m}{m(m-1)} d\mu \\&= \frac{1}{m} \int_X u^{m-1} Lu^m d\mu + \frac{1}{m(m-1)} \int_X u^m (x) \sum_{y \in X} k(x,y)  \Lambda_H \big(u^m(y),u^m(x)\big) d\mu(x),
\end{align*}
which is equivalent to
\begin{align*}
\Big(1-\frac{m-1}{m}\Big) \frac{d}{dt} &\int_X \frac{u^m}{m(m-1)} d\mu = \frac{1}{m(m-1)} \int_X u^m (x) \sum_{y \in X} k(x,y)  \Lambda_H \big(u^m(y),u^m(x)\big) d\mu(x).
\end{align*}
Hence,
\begin{align*}
\frac{d}{dt} &\int_X \frac{u^m}{m(m-1)} d\mu = \frac{1}{m-1} \int_X u^m (x) \sum_{y \in X} k(x,y)  \Lambda_H \big(u^m(y),u^m(x)\big) d\mu(x) \\ 
&= -\frac{1}{m} \int_X u(x) \Big(-\frac{m}{m-1} u^{m-1} (x) \sum_{y \in X} k(x,y) \Big[mu^{m-1}(y)-mu^{m-1}(x) \\ & \qquad \qquad \qquad \qquad -(m-1)\big(\frac{u^m(y)}{u(x)}-u^{m-1}(x)\big)\Big] \Big) d \mu(x)\\
 &= -\frac{1}{m} \int_X u(x) \sum_{y \in X} k(x,y) \Big(m u^{m-2}(x)u^m(y) +\frac{m}{m-1} u^{2m-2} (x) \\
  & \qquad \qquad \qquad-\frac{m^2}{m-1} u^{m-1}(x)u^{m-1}(y)\Big) d\mu(x) \\ 
  &= -\frac{1}{m} \int_X u \tilde{\Psi}_\Upsilon^{(m)} (v) d\mu.
\end{align*}
\end{proof}
\begin{bemerkung}
Equation (\ref{DiscreteChainRuleV}) is the evolution equation for $v$ in the discrete case. Comparing (\ref{DiscreteChainRuleV}) and (\ref{ChainRuleV}), we see that the term $\tilde{\Psi}_{\Upsilon}^{(m)} (u)$ replaces $\vert \nabla v \vert^2$ occurring in the diffusion setting. 

The same appears in the dissipation term of the R\'{e}nyi entropy. Indeed, in the diffusion setting (e.g.\ with a bounded smooth domain $X\subset \iR^d$ and homogeneous Neumann boundary condition) we (formally) have
\begin{align*}
\frac{d}{dt} \int_X \frac{u^m}{m(m-1)} d\mu &= \frac{1}{m-1} \int_X u^{m-1} \Delta u^m d\mu = \frac{m}{m-1} \int_X u^{m-1} \nabla \cdot \big(u^{m-1} \nabla u \big) d\mu \\ &= -\frac{m}{m-1} \int_X \nabla \big(u^{m-1}\big) u^{m-1}\cdot \nabla u d\mu \\ &= -\frac{1}{m} \int_X u m^2 u^{2m-4} \vert \nabla u \vert^2 d\mu \\ &= -\frac{1}{m} \int_X u \vert \nabla v \vert^2 d\mu,
\end{align*}
by integration by parts. When comparing this with (\ref{DiscreteEntropyEqu}), we also observe that $\vert \nabla v \vert^2$ is replaced by $\tilde{\Psi}_{\Upsilon}^{(m)} (u)$ in the discrete setting.
\end{bemerkung}

In the following lemma we state some important properties of $\tilde{\Psi}_{\Upsilon}^{(m)}$. 

\begin{lemma}
\label{LemmaEigTildePsi}
Let $m>1$, $u: X\to (0,\infty)$, $v$ the corresponding pressure, and $x\in X$. Then the following statements hold true.
\begin{enumerate}
\item[(i)] $\tilde{\Psi}_{\Upsilon}^{(m)} (u) (x) \geq 0$. 
\item[(ii)] $\lim_{m \to 1} \tilde{\Psi}_{\Upsilon}^{(m)} (v)(x) = \Psi_\Upsilon (\log u)(x)$.
\item[(iii)] In the special case $m=2$,
\begin{align*}
\tilde{\Psi}_{\Upsilon}^{(2)} (v)(x) = \Gamma (v)(x),
\end{align*}
where $\Gamma (v)$ is the carr\'{e} du champ operator associated with $L$, see \eqref{Gammadef} below.
\end{enumerate}
\end{lemma}

\begin{proof}
(i) By convexity of $\Upsilon $ we have
\begin{align*}
\tilde{\Upsilon} (x) &= (m-1) \Big( \frac{m-1}{m} \Upsilon \big(\frac{m}{m-1} x \big)- \Upsilon (x) \Big) \\ &= (m-1) \Big( \frac{m-1}{m} \Upsilon \big(\frac{m}{m-1} x \big) + \frac{1}{m} \Upsilon (0) - \frac{1}{m} \Upsilon (0)- \Upsilon (x) \Big) \\ &\geq (m-1) \Big( \Upsilon (x) - \frac{1}{m} \Upsilon (0)- \Upsilon (x) \Big) =0
\end{align*}
and thus $\tilde{\Psi}_{\Upsilon}^{(m)} (u)(x) = u^2(x)\Psi_{\tilde{\Upsilon}} (\log u)(x) \geq 0$. 

(ii) Using l'H\^{o}spital's rule we find
\begin{align*}
\lim_{m \to 1} &\tilde{\Psi}_{\Upsilon}^{(m)} (v)(x) \\&= \sum_{y \in X} k(x,y) \frac{u(y)}{u(x)} - \lim_{m \to 1} \sum_{y \in X} k(x,y) \Big(\frac{m^2 u^{m-1} (y)u^{m-1} (x) - m u^{2m-2}(x)}{m-1} \Big) \\ &= \sum_{y \in X} k(x,y) \frac{u(y)}{u(x)} - \sum_{y \in X} k(x,y) \Big(1+\log \big(u(y)u(x)\big)-2 \log \big(u(x)\big) \Big) \\ &= \Psi_\Upsilon (\log u)(x).
\end{align*}

(iii) It follows directly from the definition of $\tilde{\Psi}_{\Upsilon}^{(m)}$ that
\begin{align} \label{Gammadef}
\tilde{\Psi}_{\Upsilon}^{(2)} (v)(x) = \frac{1}{2} \sum_{y \in X} k(x,y) \big(v(y)-v(x)\big)^2,
\end{align}
which coincides with the definition of the carr\'{e} du champ operator $\Gamma (v)$, see e.g.\ \cite[Section 1]{DKZ}
and \cite{BGL}.
\end{proof}

\begin{bemerkung}
(i) As $\partial_t v \to \partial_t \log u \ (m \to 1)$ and $(m-1) v Lv \to L \log u \ (m \to 1)$, we see, together with Lemma \ref{LemmaEigTildePsi} (ii), that (\ref{DiscreteChainRuleV}) yields the evolution equation
\begin{align*}
\partial_t \log u = L \log u + \Psi_\Upsilon (\log u)
\end{align*}
as $m \to 1$. As already mentioned, this equation for $\log u$ is the key starting point in \cite{DKZ} and \cite{KWZ} 
for deriving Li-Yau and Harnack estimates for positive solutions to the (discrete) heat equation.

(ii) The fact that $\tilde{\Psi}_{\Upsilon}^{(2)} (v) = \Gamma (v) $ is remarkable as in the diffusion setting
with $L=\Delta$ there holds $\Gamma (v) = \vert \nabla v \vert^2$, see \cite[Chapter 1.4.2]{BGL}. In this sense,
the continuous and the discrete evolution equation for $v$ coincide in the special case $m=2$.
\end{bemerkung}

\section{$CD$-condition and positive examples}
The observations from the previous chapter will be important later in the derivation of discrete Aronson-B\'{e}nilan estimates and related differential Harnack inequalities. In this section we introduce a curvature-dimension condition, denoted by 
$CD_m (0,d)$, which leads to an Aronson-B\'{e}nilan estimate on finite graphs. Our approach is inspired 
by the article \cite{DKZ}. As we will see with several examples, the calculations to verify $CD_m (0,d)$
are much more involved than in the linear case $m=1$. That our $CD$-condition is indeed useful will be shown by
positive examples such as, e.g., complete graphs. Interestingly, chain-like graphs with at least 5 vertices (in particular the
unweighted lattice $\iZ$) fail to satisfy the condition $CD_m (0,d)$ for all $m\ge 2$ and $d>0$, see Chapter \ref{SectionChains}. Moreover, our calculations indicate that there is a significant difference between the cases $m\ge 2$
 and $m\in (1,2)$.

In what follows, $X$ is a countable set and $k: X \times X \to [0,\infty)$ is nontrivial and such that for any $x\in X$ the set
of all $y\in X$ with $k(x,y)>0$ is finite, i.e.\ the induced graph is locally finite.
\subsection{Curvature-dimension condition}
For $m>1$ and $u: X\to (0,\infty)$ we set 
\begin{equation}
\mathcal{D}_{m} (u)(x) = m \sum_{y \in X} k(x,y) \big(u^{m-2} (y) Lu^m(y)-u^{m-2} (x) Lu^m(x)\big),\quad x\in X.
\label{MathcalD}
\end{equation}
This operator plays a key role in the $CD_{m} (0,d)$-condition, which is defined as follows.
\begin{definition}
\label{DefCD}
Let $m>1$ and $d>0$. We say that the operator $L$ from (\ref{eq:operator}) satisfies the condition $CD_{m} (0,d)$ at $x \in X$, if for every function $u: X \to (0,\infty)$ such that the function $v = \frac{m}{m-1} u^{m-1}$ satisfies
\begin{align}
-L v (x) >0 \text{ and } -L v (x) \geq -L v (y)\quad \text{ for all } y\in X \;\text{ with } k(x,y)>0
\label{MaxProp}
\end{align}
there holds
\begin{equation}
\mathcal{D}_m (u)(x) \geq \frac{1}{d} \big(-L v (x)\big)^2.
\label{CD-Bedingung}
\end{equation}
We say that $L$ satisfies $CD_m (0,d)$, if it satisfies $CD_m(0,d)$ at any $x \in X$.
\end{definition}
\begin{bemerkung}
\label{BemerkungCDBed}
As already mentioned, the condition $CD_m(0,d)$ is inspired by the approach from \cite{DKZ}. It is motivated by the
fact that for positive solutions $u$ of the discrete PME we have $\partial_t Lv=\mathcal{D}_m (u)$, see the proof of
Theorem \ref{ThmAronsonBen} in the special case $\alpha=0$. In \cite{DKZ}, the authors introduced 
the condition $CD(F;0)$ (see \cite[Definition 3.8]{DKZ}), where the quadratic term has been replaced by a term $F(-Lv)$ involving a so-called $CD$-function $F$ (cf.\ \cite[Definition 3.1]{DKZ}). Taking the limit $m\to 1$ in (\ref{CD-Bedingung}), we recover the condition $CD(F;0)$ with the $CD$-function $F (x) = \frac{1}{d} x^2$. Indeed,
\begin{align*}
\mathcal{D}_1 (u)(x) = L \Psi_{\Upsilon^\prime} (\log u)(x).
\end{align*}
Thus, $CD_m(0,d)$ generalises $CD(F;0)$ for quadratic $CD$-functions $F$ to the nonlinear case $m>1$.
Recall that, as already mentioned in the introduction, in case of the discrete PME, square functions are the natural choice for
the $CD$-function.
\end{bemerkung}
\subsection{Complete graphs}
We now want to illustrate the condition $CD_m(0,d)$ with the example of an unweighted complete graph, that is, for all
$x,y$ in the finite set of vertices $X$ we have $k(x,y)=1$ whenever $x\neq y$. 

We begin with the simplest case with only two vertices.
\begin{beispiel}
\label{Bsp2-Punkte}
Let $X = \{x_1,x_2\}$ and $u: X\to (0,\infty)$. Then 
\begin{align*}
L u(x) = u(\tilde{x})-u(x), \ x \in X,
\end{align*}
where $\tilde{x}_1 = x_2$ and vice versa. Now suppose that $u$ satisfies (\ref{MaxProp}) at $x_1$. This is equivalent
to $v(x_1)>v(x_2)$, which in turn is equivalent to $u(x_1)>u(x_2)$. We have 
\begin{align*}
\mathcal{D}_m (u)(x_1) &= m \big(u^{m-2} (x_2) Lu^m(x_2)-u^{m-2} (x_1) Lu^m(x_1)\big) \\ &= m \big( u^{m-2}(x_2) u^m(x_1) -u^{2m-2}(x_2)-u^{m-2}(x_1)u^m(x_2)+u^{2m-2}(x_1)\big).
\end{align*}
Further,
\begin{align*}
\big(-L v(x_1)\big)^2 &= v^2(x_2) -2 v(x_2) v(x_1) + v^2 (x_1) \\&= \frac{m^2}{(m-1)^2} \big( u^{2m-2} (x_2) -2 u^{m-1} (x_2) u^{m-1} (x_1) + u^{2m-2} (x_1)\big),
\end{align*}
and thus setting $z = \frac{u(x_2)}{u(x_1)}\in (0,1)$ we find that $CD_m(0,\nu)$ is satisfied if and only if $f_{\nu,m} (z) \geq 0$ for all $z\in (0,1)$, where
\begin{align*}
f_{\nu,m} (z) = \nu m (z^{m-2} -z^{2m-2}-z^m+1) - \frac{m^2}{(m-1)^2} \big( z^{2m-2} -2 z^{m-1} + 1\big).
\end{align*}

We first consider the case $m>2$. Sending $z\to 0$ shows that $\nu \geq \frac{m}{(m-1)^2}$ must hold. We now show that the choice $\nu_* = \frac{m}{(m-1)^2}$ is sufficient, i.e. $f_{\nu_*,m} (z) \geq 0, z \in (0,1),$ and thus optimal for $m>2$. We have
\begin{align*}
f_{\nu_*,m} (z) = \frac{m^2}{(m-1)^2} (z^{m-2}-2z^{2m-2}-z^m+2z^{m-1})
\end{align*}
which is greater or equal to $0$ if and only if $\frac{1}{z} -2z^{m-1}-z+2 \geq 0$. The latter obviously holds for all $z\in (0,1)$. Note that this estimate is also true for any $m \in (1,2]$ but can be sharpened for these choices of $m$. Indeed, consider for example the case $m=2$. Then $f_{\nu,2} (0) = 4 \nu - 4$ and thus $\nu \geq 1$ must hold. By a straightforward calculation, one can show that $f_{1,2} (z) \geq 0, z \in [0,1],$ and thus $\nu =1 (< \frac{m}{(m-1)^2}=2)$ is the optimal choice in this case. In the case $m \in (1,2)$, we have $f_{\nu,m} \geq 0$ for all
$z\in (0,1)$ if and only if
\begin{align*}
g_m(z) := \frac{m}{(m-1)^2} \frac{z^{2m-2} -2 z^{m-1} + 1}{z^{m-2}-z^{2m-2}-z^m+1} = \frac{m}{(m-1)^2} \frac{(1-z^{m-1})^2}{(1+z^{m-2})(1-z^m)} \leq \nu, \ z \in (0,1),
\end{align*}
by positivity of the term $z^{m-2}-z^{2m-2}-z^m+1=(1+z^{m-2})(1-z^m) , \ z \in [0,1]$. Thus, the optimal constant $\nu$ in this case is given by the maximum $\nu = \max_{x \in (0,1)} g_m(x) \ \big(<\frac{m}{(m-1)^2}\big)$, which is indeed assumed in $(0,1)$.
\end{beispiel}

\begin{beispiel}
Next, we consider the case with 3 vertices, i.e. $X = \{x_1,x_2,x_3\}$ and $k(x_i,x_j) =1, \ i \neq j$. Let $u$ be a positive function on $X$ satisfying (\ref{MaxProp}) at $x_1$. Observe that since
\begin{align*}
L u(x_i) = \sum_{j \neq i} u(x_j)-2u(x_i), \quad i=1,2,3,
\end{align*}
the property $-Lv(x_1)\ge -Lv(x_i)$, $i=2,3$, is equivalent to $u(x_1) \geq u(x_i), i = 2,3$. We have 
\begin{align*}
\mathcal{D}_m &(u)(x_1) = m \big(u^{m-2} (x_2) Lu^m(x_2)+ u^{m-2} (x_3) Lu^m(x_3)-2u^{m-2} (x_1) Lu^m(x_1)\big) \\ &= m \Big( u^{m-2}(x_2)\big(u^m(x_1)+u^m(x_3)-2u^m(x_2)\big)+u^{m-2}(x_3) \big(u^m(x_2)+u^m(x_1)-2u^m(x_3)\big)\\& \qquad \qquad -2u^{m-2}(x_1) \big(u^m(x_2)+u^m(x_3)-2u^m(x_1)\big)\Big).
\end{align*}
Further,
\begin{align*}
\big(-L v(x_1)\big)^2 &= v^2(x_2)+v^2(x_3)+4v^2(x_1)+2v(x_2)v(x_3)-4v(x_1)v(x_2)-4v(x_1)v(x_3) \\ &= \frac{m^2}{(m-1)^2} \Big(u^{2m-2} (x_2)+u^{2m-2} (x_3)+4u^{2m-2} (x_1)+2u^{m-1} (x_2)u^{m-1} (x_3)\\ &\qquad \qquad -4u^{m-1} (x_1) u^{m-1} (x_2)-4u^{m-1} (x_1) u^{m-1} (x_3)\Big),
\end{align*}
and thus setting $z_i = \frac{u(x_{i+1})}{u(x_1)}\in (0,1], \ i=1,2$, we find that $CD_m(0,\nu)$ is satisfied if and only if $f_{\nu,m} (z_1,z_2) \geq 0$, where
\begin{align*}
f_{\nu,m} (z_1,z_2) &= \nu m \big(z_1^{m-2}z_2^m+z_2^{m-2}z_1^m+z_1^{m-2}+z_2^{m-2}-2z_1^{2m-2}-2z_2^{2m-2}-2z_1^m-2z_2^m+4\big)\\ &\qquad - \frac{m^2}{(m-1)^2} \big((z_1^{m-1}+z_2^{m-1})^2-4z_1^{m-1}-4z_2^{m-1}+4\big),\quad z_1,z_2\in (0,1].
\end{align*}

Again, we first consider the case $m>2$. Sending $z_1,z_2\to 0$ shows that $\nu \geq \frac{m}{(m-1)^2}$ must hold. As in Example \ref{Bsp2-Punkte}, the choice $\nu_* = \frac{m}{(m-1)^2}$ is also sufficient, i.e. $f_{\nu_*,m} (z_1,z_2) \geq 0, \ z_1,z_2 \in (0,1],$ and thus optimal for $m>2$. To see this, we calculate
\begin{align*}
f_{\nu_*,m} (z_1,z_2) &= \frac{m^2}{(m-1)^2} \big(z_1^{m-2}z_2^m+z_2^{m-2}z_1^m+z_1^{m-2}+z_2^{m-2}-2z_1^{2m-2}-2z_2^{2m-2}-2z_1^m-2z_2^m\\ &\qquad \qquad \qquad \qquad -(z_1^{m-1}+z_2^{m-1})^2+4z_1^{m-1}+4z_2^{m-1}\big).
\end{align*}
Now,
\begin{align*}
&\frac{(m-1)^2 f_{\nu_*,m} (z_1,z_2)}{m^2 (z_1^{m-1}+z_2^{m-1})} = \frac{z_1^{m-2}z_2^m+z_2^{m-2}z_1^m}{z_1^{m-1}+z_2^{m-1}}+\frac{z_1^{m-2}+z_2^{m-2}}{z_1^{m-1}+z_2^{m-1}}-2\frac{z_1^{2m-2}+z_2^{2m-2}}{z_1^{m-1}+z_2^{m-1}}\\ &\qquad \qquad \qquad \qquad \qquad \qquad-2\frac{z_1^m+z_2^m}{z_1^{m-1}+z_2^{m-1}} -z_1^{m-1}-z_2^{m-1}+4 \\ &\geq \frac{z_1^{m-2}z_2^m+z_2^{m-2}z_1^m}{z_1^{m-1}+z_2^{m-1}}-z_1^{m-1}-z_2^{m-1}+1 \geq \frac{z_1^{m-2}z_2^m+z_2^{m-2}z_1^m -2z_1^{m-1}z_2^{m-1}}{z_1^{m-1}+z_2^{m-1}} \\ &= \frac{z_1^{m-2}z_2^{m-2} (z_1-z_2)^2}{z_1^{m-1}+z_2^{m-1}} \geq 0.
\end{align*}

As before, this estimate also holds for any $m \in (1,2]$ but can be sharpened for these choices of $m$. Taking again $m=2$, we have $f_{\nu,2} (0,0) = 12 \nu - 16$ and thus $\nu \geq \frac{4}{3}$ must hold. By a straightforward calculation, one can show that $f_{\frac{4}{3},2} (z_1,z_2) \geq 0, z_1,z_2 \in (0,1],$ and thus $\nu =\frac{4}{3} (< \frac{m}{(m-1)^2}=2)$ is the optimal choice in this case, see also the end of Example \ref{BspVollstGraph}.
\end{beispiel}

\begin{beispiel}
\label{BspVollstGraph}
The technique from the previous two examples can be generalized to the case of an unweighted complete graph with $D$ vertices ($D \in \N$, $D\ge 2$), i.e. $X = \{x_1,\dots,x_D\}$ and $k(x_i,x_j) =1,\ i \neq j$. Let $u$ be a positive function satisfying (\ref{MaxProp}) at $x_1$, which in this case means that $u(x_1) \geq u(x_i), \ i = 2,\dots,D$, as
\begin{align*}
L v(x_i) = \sum_{j \neq i} v(x_j)-(D-1) v(x_i), \ i=1,\dots,D,
\end{align*}
and by monotonicity of $\{y\to y^m\}$ on $(0,\infty)$. We have
\begin{align*}
\mathcal{D}_m (u)(x_1) &= m \sum_{j \neq 1} \big(u^{m-2} (x_j) Lu^m(x_j)\big)-m(D-1)u^{m-2} (x_1) Lu^m(x_1) \\ &= m \sum_{j \neq 1} u^{m-2}(x_j) \sum_{l \neq j} u^m(x_l)-m (D-1) \sum_{j\neq 1} u^{2m-2} (x_j) \\ &\qquad  \qquad- m(D-1) u^{m-2}(x_1) \sum_{j\neq 1} u^m(x_j) +m(D-1)^2 u^{2m-2} (x_1). 
\end{align*}
Further,
\begin{align*}
\big(-L v(x_1)\big)^2 &= \Big(\sum_{j \neq 1} \big(v(x_j)\big)-(D-1)v(x_1)\Big)^2\\ &= \frac{m^2}{(m-1)^2} \Big(\big(\sum_{j\neq 1} u^{m-1}(x_j)\big)^2 -2(D-1) u^{m-1}(x_1) \sum_{j\neq 1} \big(u^{m-1}(x_j)\big) \\ & \qquad \qquad \qquad+ (D-1)^2 u^{2m-2}(x_1)\Big).
\end{align*}
Consequently, setting $z_i = \frac{u(x_{i+1})}{u(x_1)}\in (0,1], i=1,\dots,D-1$, and $z=(z_1,\ldots,z_{D-1})$ we find that $CD_m(0,\nu)$ is satisfied if and only if $f_{\nu,m} (z) \geq 0$, where
\begin{align*}
f_{\nu,m} &(z_1,\dots,z_{D-1}) \\&= \nu m \Big(\sum_{j} z_j^{m-2} \sum_{l\neq j} z_l^m + \sum_j z_j^{m-2} -(D-1) \sum_j z_j^{2m-2} -(D-1) \sum_j z_j^m +(D-1)^2\Big)\\ &\qquad  - \frac{m^2}{(m-1)^2} \Big(\big(\sum_j z_j^{m-1}\big)^2 -2(D-1) \sum_j z_j^{m-1} +(D-1)^2\Big),\quad z_1,\ldots,z_{D-1}\in (0,1].
\end{align*}

Again, we first consider the case $m>2$. Sending all $z_i\to 0$ shows that $\nu \geq \frac{m}{(m-1)^2}$ is necessary. As before, the choice $\nu_* = \frac{m}{(m-1)^2}$ is also sufficient in this case, i.e. $f_{\nu_*,m} (z_1,\dots,z_{D-1}) \geq 0, z_1,\dots,z_{D-1} \in (0,1],$ and thus optimal for $m>2$. In fact, we calculate
\begin{align*}
f_{\nu_*,m} (z_1,\dots,z_{D-1}) &= \frac{m^2}{(m-1)^2} \Big(\sum_{j} z_j^{m-2} \sum_{l\neq j} z_l^m + \sum_j z_j^{m-2} -(D-1) \sum_j z_j^{2m-2} \\ & \qquad  -(D-1) \sum_j z_j^m-\big(\sum_j z_j^{m-1}\big)^2 +2(D-1) \sum_j z_j^{m-1}\Big).
\end{align*}
Now,
\begin{align*}
\frac{(m-1)^2  f_{\nu_*,m} (z)}{m^2 \sum_j z_j^{m-1}} &= \frac{\sum_{j} z_j^{m-2} \sum_{l\neq j} z_l^m}{\sum_j z_j^{m-1}}+\frac{\sum_j z_j^{m-2}}{\sum_j z_j^{m-1}}-(D-1)\frac{\sum_j z_j^{2m-2}}{\sum_j z_j^{m-1}}\\ &\qquad  \qquad -(D-1)\frac{\sum_j z_j^m}{\sum_j z_j^{m-1}} -\sum_j z_j^{m-1}+2(D-1) \\ &\geq \frac{\sum_{j} z_j^{m-2} \sum_{l\neq j} z_l^m}{\sum_j z_j^{m-1}}-\sum_j z_j^{m-1}+1 \geq \frac{\sum_{j} z_j^{m-2} \sum_{l} z_l^m - \big(\sum_j z_j^{m-1} \big)^2}{\sum_j z_j^{m-1}} \geq 0,
\end{align*}
since 
\begin{align*}
\sum_{j} z_j^{m-2} \sum_{l} z_l^m - \big(\sum_j z_j^{m-1} \big)^2 &= \sum_j \sum_l (z_j^{m-2} z_l^m -z_j^{m-1} z_l^{m-1}) \\ &= \frac{1}{2} \sum_j \sum_l (z_j-z_l)^2 z_j^{m-2}z_l^{m-2} \geq 0.
\end{align*}

As before, this estimate also holds for any $m \in (1,2]$ but can be improved for these choices of $m$. In the case $m=2$ we have for example $f_{\nu,2} (0,\dots,0) = 2\nu \big((D-1)+(D-1)^2\big) -4 (D-1)^2$ and thus $\nu \geq \frac{2(D-1)}{D}$ must hold. By a straightforward calculation, one can show that 
\begin{align*}
f_{\frac{2(D-1)}{D},2} (z_1,\dots,z_{D-1}) = D \sum_j z_j - \frac{D}{2} \sum_j z_j^2 - \frac{D}{2(D-1)} \big(\sum_j z_j \big)^2, \ z_1,\dots,z_{D-1} \in (0,1].
\end{align*}
Now,
\begin{align*}
\frac{f_{\frac{2(D-1)}{D},2} (z_1,\dots,z_{D-1})}{\sum_j z_j} = D-\frac{D}{2} \frac{\sum_j z_j^2}{\sum_j z_j} - \frac{D}{2(D-1)} \sum_j z_j \geq D - \frac{D}{2} -\frac{D}{2}=0
\end{align*}
and thus $\nu =\frac{2(D-1)}{D} (< \frac{m}{(m-1)^2}=2)$ is the optimal choice in this case. Setting $D=2$ respectively $D=3$ yields the results from the previous two examples.
\end{beispiel}

\subsection{The square for $m=2$}
\label{SectionSquare}
In this subsection we consider as another example the unweighted square in the case $m=2$, i.e. $X= \{x,y_1,y_2,z\}$ and the (symmetric) kernel $k$ is defined by $k(x,y_i) = k(z,y_i)=1, \ i=1,2$ and $k(x,z)=k(y_1,y_2)=0$.
\begin{center}
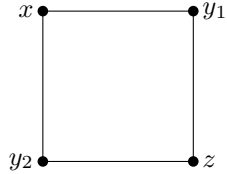
\begin{figure}[h]
 \begin{tikzpicture}
\coordinate[label=left:$y_2$] (y_2) at (0,0);
\coordinate[label=right:$z$] (z) at (2,0);
\coordinate[label=right:$y_1$] (y_1) at (2,2);
\coordinate[label=left:$x$] (x) at (0,2);
  \draw (x) -- (y_1);
  \draw (x) -- (y_2);
  \draw (y_2) -- (z);
  \draw (y_1) -- (z);
\fill (x) circle (2pt);
\fill (y_1) circle (2pt);
\fill (z) circle (2pt);
\fill (y_2) circle (2pt);
\end{tikzpicture}
\caption{The square.}
\end{figure}
\end{center}

Let $u$ be a positive function on $X$ and $v=2u$ the associated pressure function. Suppose that $v$ satisfies the maximum property (\ref{MaxProp}) at $x$. This is equivalent to the following conditions:
\begin{align*}
(\text{\Romannum{1}}): \ &-Lv(x) > 0 \Leftrightarrow 2 v(x) > v(y_1)+v(y_2) \\
(\text{\Romannum{2}}): \ &-Lv(x) \geq -Lv(y_1) \Leftrightarrow v(z) \geq 3 v(y_1)-3v(x)+v(y_2) \\
(\text{\Romannum{3}}): \ &-Lv(x) \geq -Lv(y_2) \Leftrightarrow v(z) \geq 3 v(y_2)-3v(x)+v(y_1).
\end{align*}
Define $a := \frac{v(y_1)}{v(x)} (= \frac{u(y_1)}{u(x)})$ and $b := \frac{v(y_2)}{v(x)} (= \frac{u(y_2)}{u(x)})$. Then $(\text{\Romannum{1}}) - (\text{\Romannum{3}})$ is equivalent to
\begin{align*}
(\text{\Romannum{1}})^\prime: \ &2 > a+b \\
(\text{\Romannum{2}})^\prime: \ &\frac{u(z)}{u(x)}  \geq 3 a-3+b \\
(\text{\Romannum{3}})^\prime: \ &\frac{u(z)}{u(x)} \geq 3 b-3+a.
\end{align*}
We want to determine the (minimal) $d>0$ such that $\mathcal{D}_2 (u)(x) \geq \frac{1}{d} \big(-Lv(x)\big)^2$. This is equivalent to 
\begin{align*}
\frac{\mathcal{D}_2 (u)(x)}{u^2(x)} \geq \frac{1}{d} \frac{\big(-Lv(x)\big)^2}{u^2(x)}.
\end{align*}
We have 
\begin{align*}
\frac{\mathcal{D}_2 (u)(x)}{u^2(x)} = \frac{1}{u^2(x)} \Big(2 Lu^2 (y_1)+2Lu^2(y_2)-4 Lu^2(x)\Big) = 12-8a^2-8b^2+4 \Big( \frac{u(z)}{u(x)}\Big)^2
\end{align*}
and
\begin{align*}
\frac{\big(-Lv(x)\big)^2}{u^2(x)} = \frac{1}{u^2(x)} \Big(4u(x)-2u(y_1)-2u(y_2)\Big)^2 = 4(2-a-b)^2.
\end{align*}
In order to apply $(\text{\Romannum{1}})^\prime-(\text{\Romannum{3}})^\prime$, we have to distinguish 4 cases. \\

\emph{Case 1:} $3a-3+b \geq 0$ and $3b-3+a \geq 0$. Then by (\Romannum{2})$^\prime$ and (\Romannum{3})$^\prime$ we find
\begin{align*}
\frac{\mathcal{D}_2 (u)(x)}{u^2(x)} &\geq 12-8a^2-8b^2+2(3a-3+b)^2+2(3b-3+a)^2 \\ &=48 -48a-48b+12a^2+12b^2+24ab = 12(2-a-b)^2 = 3 \frac{\big(-Lv(x)\big)^2}{u^2(x)}.
\end{align*}

\emph{Case 2:} $3a-3+b < 0$ and $3b-3+a < 0$. Then $a<1$ and $b<1$ and thus
\begin{align*}
\frac{\mathcal{D}_2 (u)(x)}{u^2(x)} &\geq 12-8a^2-8b^2 \\ &=12a+12b-11a^2-11b^2-6ab+3(4-4a-4b+2ab+a^2+b^2) \\ &= 12a+12b-11a^2-11b^2-6ab+3(2-a-b)^2 \\ &\geq 12a-11a^2-3a(3-3a)+12b-11b^2-3b(3-3b)+3(2-a-b)^2\\&= a(3-2a)+b(3-2b)+3(2-a-b)^2 \\& \geq 3 (2-a-b)^2 = \frac{3}{4} \frac{\big(-Lv(x)\big)^2}{u^2(x)}.
\end{align*}
Note, that we have equality in the limit case when $a$ and $b$ go to zero.

\emph{Case 3:} $3a-3+b \geq 0$ and $3b-3+a < 0$. Then we have $a \geq b$. We show that $d=\frac{4}{3}$ is also sufficient in this case.
\begin{align*}
\frac{\mathcal{D}_2 (u)(x)}{u^2(x)} - \frac{3}{4} \frac{\big(-Lv(x)\big)^2}{u^2(x)} &\geq 12-8a^2-8b^2+4(3a-3+b)^2-{3} (2-a-b)^2 \\ &= 36-60a-12b+25a^2-7b^2+18ab \\ &\geq 36-60a-12b+25a^2+b^2+10ab \\ &= (5a+b-6)^2 \geq 0.
\end{align*}

\emph{Case 4:} $3a-3+b \geq 0$ and $3b-3+a < 0$. Analogeously to Case 3.\\

Summarizing all 4 cases, we find that
\begin{align*}
\mathcal{D}_2 (u)(x) \geq \frac{3}{4} \big(-Lv(x)\big)^2
\end{align*}
with equality in the limit case $u(y_1)=u(y_2)=u(z)=0$. Thus, $CD_2 (0,d)$ is satisfied at $x$ with the optimal constant $d=\frac{4}{3}$, and by symmetry, we conclude that $L$ satisfies $CD_2 (0,\frac{4}{3})$.

\begin{bemerkung}
As we will see in Theorem \ref{ThmAronsonBen}, this estimate will be sufficient for proving an Aronson-B\'{e}nilan estimate for solutions to the PME on the square in the case $m=2$. In Corollary \ref{KorollarSquare}, we will further show the validity of an alternative condition that gives a little weaker differential estimate as the one of Aronson and B\'{e}nilan but instead is valid for any $m >1$.
\end{bemerkung}
 
\section{Chain-like graphs, the lattice $\Z$, and a generalised $CD$-condition}
\label{SectionChains}
\label{SubsecZ}
We will now study chain-like graphs in the unweighted case. We already saw in Example \ref{Bsp2-Punkte} that for any $m>1$ there exists a $d>0$ such that the 2-point graph satisfies $CD_m (0,d)$. In this section, we will see that the condition $CD_m (0,d)$ fails on chain structures at least in the case $m =2$ whenever the chain has more than 2 vertices. For chains with more than 4 vertices we even prove the failure for any $m \geq 2$ in Example \ref{Bspp>1}. As the same problem also occurs on the lattice $\Z$, we can conclude that there exists no $d>0$ such that the discrete Laplacian on the lattice $\Z$ satisfies $CD_m(0,d)$ for $m \geq 2$. 

However, motivated by the failure of $CD_m (0,d)$, we introduce a more general curvature-dimension condition, denoted by $CD_{m,\alpha} (0,d)$ and with an additional parameter $\alpha\in [0,1]$, and will show that the discrete Laplacian on $\Z$ satisfies the condition $CD_{m,1} (0,\frac{1}{m-1})$. 

\subsection{$CD_m(0,d)$ on chain structures}
We first consider the chain with 3 vertices for $m=2$. 

\begin{beispiel}
\label{BspChain3}
Let $m=2$, $X=\{x,y,z\}$, $k(x,z) = k(x,y)=1$ and $k(y,z)=0$. Let $L$ be the operator generated by $k$. 

\begin{center}
\begin{figure}[h]
 \begin{tikzpicture}
\coordinate[label=above:$x$] (x) at (0,0);
\coordinate[label=above:$z$] (z) at (2,0);
\coordinate[label=above:$y$] (y) at (-2,0);
  \draw (x) -- (y);
  \draw (x) -- (z);
\fill (x) circle (2pt);
\fill (y) circle (2pt);
\fill (z) circle (2pt);
\end{tikzpicture}
\caption{Chain of 3 vertices.}
\end{figure}
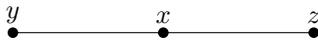
\end{center}

Observe that by continuity and since $m=2$, validity of the $CD_2(0,d)$ condition (which is a property for positive
functions) extends to nonnegative functions $u:X\to [0,\infty)$ in (\ref{MaxProp}). Now, let $u(x) = 1$, $u(y) =1.5$ and $u(z)=0$. Then $-L v(x) = -2Lu(x) = 1>0$, $-Lv(y) = 1 \leq -Lv(x)$ and $-Lv(z) = -2 \leq -Lv(x)$ and thus (\ref{MaxProp}) is satisfied at $x$. However, for $\mathcal{D}_2 (u) (x)$ we obtain
\begin{align*}
\mathcal{D}_2 (u)(x) = 2Lu^2(y)+2Lu^2(z)-4Lu^2(x) = 12 u^2(x)-6u^2(y)-6u^2(z) = -\frac{3}{2} <0,
\end{align*}
and thus there is no $d>0$ such that $\mathcal{D}_2 (u)(x) \geq \frac{1}{d} \big(-Lv(x)\big)^2$ holds, i.e.\ $CD_2(0,d)$ cannot be satisfied for any $d>0$.
\end{beispiel}

The same problem arises when considering the chain with 4 vertices.

\begin{beispiel}
\label{BspChain4}
Let $m=2$, $X=\{w,x,y,z\}$, $k(w,x) = k(x,y)=k(y,z)=1$ and $k=0$ for all other edges. Let $L$ be the operator generated by $k$. 
\begin{center}
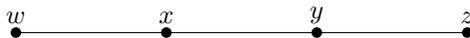
\begin{figure}[h]
 \begin{tikzpicture}
\coordinate[label=above:$x$] (x) at (-1,0);
\coordinate[label=above:$z$] (z) at (3,0);
\coordinate[label=above:$y$] (y) at (1,0);
\coordinate[label=above:$w$] (w) at (-3,0);
  \draw (x) -- (y);
  \draw (y) -- (z);
    \draw (w) -- (x);
\fill (w) circle (2pt);
\fill (x) circle (2pt);
\fill (y) circle (2pt);
\fill (z) circle (2pt);
\end{tikzpicture}
\caption{Chain of 4 vertices.}
\end{figure}
\end{center}
Let $u(x) = 1$, $u(w) =1.5$ and $u(y)=u(z)=0$. Then $-L v(x) = -2Lu(x) = 1>0$, $-Lv(w) = 1 \leq -Lv(x)$ and $-Lv(y) = -2 \leq -Lv(x)$ and thus (\ref{MaxProp}) is satisfied at $x$. However, for $\mathcal{D}_2 (u) (x)$ we find
\begin{align*}
\mathcal{D}_2 (u)(x) = 2Lu^2(w)+2Lu^2(y)-4Lu^2(x) = 12 u^2(x)-6u^2(w)-8u^2(y) +2u^2(z) = -\frac{3}{2} <0,
\end{align*}
and therefore $CD_2(0,d)$ cannot be satisfied for any $d>0$.
\end{beispiel}

We finally consider the chain of 5 vertices and show similar results as before. The point here is that $CD_m (0,d)$ also fails to be true at the point in the middle of the chain. We begin with the case $m=2$ and generalize our negative result afterwards to the case $m\geq 2$. As a consequence, we see that for every chain of more than 5 vertices $CD_m (0,d)$ is violated for $m\geq 2$. 
\begin{beispiel}
\label{Bspp=2}
Let $m=2$, $X=\{v,w,x,y,z\}$, $k(v,w)=k(w,x) = k(x,y)=k(y,z)=1$ and $k=0$ for all remaining edges. Let $L$ be the operator associated with $k$. 

\begin{center}
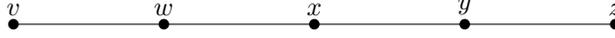
\begin{figure}[h]
 \begin{tikzpicture}
\coordinate[label=above:$x$] (x) at (0,0);
\coordinate[label=above:$z$] (z) at (4,0);
\coordinate[label=above:$y$] (y) at (2,0);
\coordinate[label=above:$w$] (w) at (-2,0);
\coordinate[label=above:$v$] (v) at (-4,0);
  \draw (x) -- (y);
  \draw (y) -- (z);
    \draw (w) -- (x);
    \draw (v) -- (w);
\fill (v) circle (2pt);
\fill (w) circle (2pt);
\fill (x) circle (2pt);
\fill (y) circle (2pt);
\fill (z) circle (2pt);
\end{tikzpicture}
\caption{Chain of 5 vertices.}
\end{figure}
\end{center}

For $\varepsilon \in \big(0,\frac{3}{5}\big]$ we define the function $u_\varepsilon: X \to [0,\infty)$ via the function $v_\varepsilon$ that is given by
\begin{align*}
v_\varepsilon (l) = \begin{cases}
1 &, l =x,\\
2-2\varepsilon &, l=y,\\
\varepsilon&, l =w,\\
3-5\varepsilon &, l=z,\\
0 &, l=v.
\end{cases}
\end{align*} 
We set $u_\varepsilon = \frac{1}{2} v_\varepsilon$. We then have $-Lv_\varepsilon (x) = -L v_\varepsilon (y) = \varepsilon$ and $-L v_\varepsilon (w) = 2\varepsilon-1 (< \varepsilon)$. Thus, (\ref{MaxProp}) is satisfied at $x$ for any $\varepsilon \in \big(0,\frac{3}{5}\big]$. For $\mathcal{D}_2 (u_\varepsilon) (x)$ we get
\begin{align*}
\mathcal{D}_2 (u_\varepsilon) (x) &= 1+(3-5 \varepsilon)^2 -2 (2-2\varepsilon)^2+1-2 \varepsilon^2 - 2 \big((2-2\varepsilon)^2+\varepsilon^2-2\big) \to -1 \ (\varepsilon \to 0),
\end{align*}
which shows that for sufficiently small $\varepsilon>0$ there does not exist any $d>0$ such that $\mathcal{D}_2 (u_\varepsilon) (x) \geq \frac{1}{d} \big( -L v_\varepsilon (x)\big)^2$, i.e. $CD_2(0,d)$ is not satisfied for any $d>0$.
\end{beispiel}

As already mentioned, the counterexample for the case $m=2$ can even be generalized to the case $m \geq 2$ as the following example shows.

\begin{beispiel}
\label{Bspp>1}
Let $L, \varepsilon$ and $v_\varepsilon$ be as in Example \ref{Bspp=2}. We now set $u_\varepsilon = \Big(\frac{m-1}{m} v_\varepsilon \Big)^\frac{1}{m-1}$. By the same arguments as in Example \ref{Bspp=2}, (\ref{MaxProp}) is satisfied at $x$ for any $\varepsilon \in \big(0,\frac{3}{5} \big]$ and for $\mathcal{D}_m (u_\varepsilon) (x)$ we get
\begin{align*}
&\mathcal{D}_m (u_\varepsilon) (x) = m \Big( u_\varepsilon^{m-2} (y) Lu_\varepsilon^m (y) + u_\varepsilon^{m-2}(w) Lu_\varepsilon^m (w) -2u_\varepsilon^{m-2} (x) Lu_\varepsilon^m(x)\Big) \\&= \frac{(m-1)^2}{m} \Big( v_\varepsilon^\frac{m-2}{m-1} (y) Lv_\varepsilon^\frac{m}{m-1} (y) + v_\varepsilon^\frac{m-2}{m-1} (w) Lv_\varepsilon^\frac{m}{m-1} (w) -2v_\varepsilon^\frac{m-2}{m-1} (x) Lv_\varepsilon^\frac{m}{m-1} (x)\Big) \\ &= \frac{(m-1)^2}{m} \Big( (2-2\varepsilon)^\frac{m-2}{m-1} \big((3-5\varepsilon)^\frac{m}{m-1} +1-2(2-2\varepsilon)^\frac{m}{m-1} \big) + \varepsilon^\frac{m-2}{m-1} \big(1-2\varepsilon^\frac{m}{m-1}\big) \\ & \qquad \qquad \qquad \qquad-2 \big( (2-2\varepsilon)^\frac{m}{m-1} + \varepsilon^\frac{m}{m-1}-2\big)\Big).
\end{align*}
As $m>2$, we find
\begin{align*}
\mathcal{D}_m (u_\varepsilon)(x) \xrightarrow{\varepsilon \to 0} \ &\frac{(m-1)^2}{m} \Big( 2^\frac{m-2}{m-1} \big(3^\frac{m}{m-1} +1-2 \cdot 2^\frac{m}{m-1}\big)-2 \big(2^\frac{m}{m-1} -2\big)\Big) \\ = \ &\frac{(m-1)^2}{m} \Big(2^\frac{m-2}{m-1} 3^\frac{m}{m-1} + 2^\frac{m-2}{m-1} -8-2 \cdot 2^\frac{m}{m-1} + 4\Big) \\ = \ & \frac{(m-1)^2}{m} 2^{-\frac{1}{m-1}} \big(2 \cdot 3^\frac{m}{m-1} +2 -2^\frac{2m}{m-1} -2^\frac{2m-1}{m-1}\big) \\ = \ & \frac{(m-1)^2}{m} 2^{-\frac{1}{m-1}} \big(2 \cdot 3^\frac{m}{m-1} -4^\frac{m}{m-1} -2^\frac{m}{m-1} +2-2^\frac{m}{m-1}\big) \\ < \ &0.
\end{align*}
Here, the last step is justified as $2 \cdot 3^\frac{m}{m-1} -4^\frac{m}{m-1} -2^\frac{m}{m-1} \leq 0$, by convexity of the mapping $x \mapsto x^\frac{m}{m-1}$ and $2-2^\frac{m}{m-1} <0$ as $\frac{m}{m-1}>1$. As in the previous example, this shows that $CD_m (0,d)$ is not valid for any $d>0$ when $m\geq 2$.
\end{beispiel}

\begin{bemerkung}
The previous example is important for the discrete Laplacian on the (unweighted) lattice $\Z$. As only neighbours of first and second order have influence on the term $\mathcal{D}_m$, we can conclude with Example \ref{Bspp>1} that the discrete Laplacian on $\Z$ does not satisfy $CD_m (0,d)$ for any $d>0$ if $m\geq 2$. The same holds for any chain of more than 5 vertices. The case $m\in (1,2)$ remains an interesting open problem. 
\end{bemerkung}

\subsection{Generalized curvature-dimension condition}
As we saw, chain-like graphs with more than 2 vertices including the discrete Laplacian on $\Z$ do not satisfy $CD_m(0,d)$ with any $d>0$ when $m \geq 2$. Looking for alternatives, we introduce a more general version of the condition $CD_m (0,d)$ involving an additional parameter. It turns out that for a special choice of the additional parameter, the discrete Laplacian on $\iZ$ does satisfy the new $CD$-condition for any $m>1$. 

Let $m>1$, $X$ a countable set and $k: X \times X \to [0,\infty)$ be nontrivial and such that for any $x\in X$ the set
of all $y\in X$ with $k(x,y)>0$ is finite. We generalize the operator $\mathcal{D}_m$ from (\ref{MathcalD}) by defining for $u:X \to (0,\infty)$ and $x\in X$,
\begin{align*}
\mathcal{D}_{m,\alpha} (u)(x) &:= \sum_{y \in X} k(x,y) \Big(\Big(1-\alpha+\alpha \frac{u(y)}{u(x)}\Big) m u^{m-2} (y) Lu^m(y) \\ &\qquad \qquad \qquad \qquad -\Big(m-\alpha+\alpha \Big(\frac{u(y)}{u(x)}\Big)^m\Big) u^{m-2} (x) Lu^m(x)\Big), \notag
\end{align*}
where $\alpha \in [0,1]$. Note that $\mathcal{D}_{m,0} = \mathcal{D}_m$. 

The generalized curvature-dimension condition $CD_{m,\alpha}(0,d)$ is defined as follows.

\begin{definition}
\label{DefCDGen}
Let $m>1$ and $\alpha \in [0,1]$. We say that the operator $L$ from (\ref{eq:operator}) satisfies the condition $CD_{m,\alpha} (0,d)$ with some $d>0$ at $x \in X$, if for every function $u: X \to (0,\infty)$ such that the function $v = \frac{m}{m-1} u^{m-1}$ satisfies
\begin{align}
-\Big(Lv + \alpha \frac{\tilde{\Psi}_\Upsilon^{(m)} (v)}{(m-1) v}\Big)(x) >0 \text{ and } -\Big(Lv + \alpha \frac{\tilde{\Psi}_\Upsilon^{(m)} (v)}{(m-1) v}\Big) (x) \geq -\Big(Lv + \alpha \frac{\tilde{\Psi}_\Upsilon^{(m)} (v)}{(m-1) v} \Big) (y) 
\label{MaxPropG}
\end{align}
for all $y$ with $k(x,y)>0$, there holds
\begin{equation}
\mathcal{D}_{m,\alpha} (u)(x) \geq \frac{1}{d} \Big(-\Big(Lv + \alpha \frac{\tilde{\Psi}_\Upsilon^{(m)} (v)}{(m-1) v}\Big)(x) \Big)^2.
\label{CDInequAlt}
\end{equation}
We say that $L$ satisfies $CD_{m,\alpha} (0,d)$, if it satisfies $CD_{m,\alpha} (0,d)$ at any $x \in X$. 
\end{definition}

\begin{bemerkung}
As already mentioned, the condition $CD_{m,\alpha} (0,d)$ is closely related to $CD_m (0,d)$. In fact, $CD_{m,0} (0,d)$ is equivalent to $CD_m(0,d)$. Thus, $CD_{m,\alpha} (0,d)$ is a generalization of $CD_m(0,d)$ and gives more flexibility when dealing for example with the discrete Laplacian on the lattice $\Z$. 

The idea behind $CD_{m,\alpha} (0,d)$ is to study how the quantity $G=Lv + \alpha \frac{\tilde{\Psi}_\Upsilon^{(m)} (v)}{(m-1) v}$ evolves. It turns out that for any positive solution $u$ of the PME we have $\partial_t G=\mathcal{D}_{m,\alpha} (u)$.
\end{bemerkung}

\subsection{A positive result for the discrete Laplacian on $\Z$}

In this section we will prove that the condition $CD_{m,1} (0,d)$ is satisfied on the unweighted lattice $\Z$. The proof will be straightforward but needs quite a few calculations.

\begin{satz}
\label{SatzAltCD}
For all $m>1$, the discrete Laplacian on the (unweighted) lattice $\Z$  satisfies $CD_{m,1} \big(0, \frac{1}{m-1}\big)$.
\end{satz}

\begin{proof}
First, we recall the definition of $\tilde{\Psi}_\Upsilon^{(m)} $ from Definition \ref{DefTildePsi}. For functions $u \colon \Z \to (0,\infty)$ and $x\in \iZ$ we have
\begin{align*}
\tilde{\Psi}_{\Upsilon}^{(m)} (u)(x) =\sum_{y \in\{ x-1,x+1\}} \Big(\frac{m-1}{m} v(x)^2 + \frac{(m-1)^2}{m} \big(v(x)\big)^\frac{m-2}{m-1} \big(v(y)\big)^\frac{m}{m-1} - (m-1) v(x) v(y)\Big). 
\end{align*}
Define $ G(x) = Lv(x) + \frac{\tilde{\Psi}_\Upsilon^{(m)} (v)}{(m-1) v} (x)$. Consider a function $u$ such that $v = \frac{m}{m-1} u^{m-1}$ satisfies (\ref{MaxPropG}) at $z\in \iZ$. We then have
\begin{align*}
-G(z) &= 2v(z)-v(z+1)-v(z-1) -\frac{1}{v(z)} \Big( \frac{2}{m} v^2(z) -v(z)v(z+1)-v(z)v(z-1) \\ & \qquad \qquad + \frac{m-1}{m} v^\frac{m-2}{m-1} (z) v^\frac{m}{m-1} (z+1) + \frac{m-1}{m} v^\frac{m-2}{m-1} (z) v^\frac{m}{m-1} (z-1)\Big)\\ 
&= \Big(2-\frac{2}{m}\Big) v(z) -\frac{m-1}{m} v^{-\frac{1}{m-1}} (z) v^\frac{m}{m-1} (z+1)-\frac{m-1}{m} v^{-\frac{1}{m-1}} (z) v^\frac{m}{m-1} (z-1), \\
-G(z+1) &= \Big(2-\frac{2}{m}\Big) v(z+1) -\frac{m-1}{m} v^{-\frac{1}{m-1}} (z+1) v^\frac{m}{m-1} (z+2) \\ & \qquad \qquad   -\frac{m-1}{m} v^{-\frac{1}{m-1}} (z+1) v^\frac{m}{m-1} (z), \\
 -G(z-1) &= \Big(2-\frac{2}{m}\Big) v(z-1) -\frac{m-1}{m} v^{-\frac{1}{m-1}} (z-1) v^\frac{m}{m-1} (z-2) \\ & \qquad \qquad  -\frac{m-1}{m} v^{-\frac{1}{m-1}} (z-1) v^\frac{m}{m-1} (z). 
\end{align*}
Thus, defining the variables $a = \frac{v(z+1)}{v(z)}$, $b = \frac{v(z-1)}{v(z)}$, $\sigma = \frac{v(z+2)}{v(z)}$ and $\nu = \frac{v(z-2)}{v(z)}$, (\ref{MaxPropG}) gives
\begin{align*}
(\text{\Romannum{1}}): \ &-G(z) > 0 \Leftrightarrow 2 - a^\frac{m}{m-1} -b^\frac{m}{m-1} >0 \\
(\text{\Romannum{2}}): \ &-G(z) \geq -G(z+1) \Leftrightarrow 2-a^\frac{m}{m-1} -b^\frac{m}{m-1} \geq 2a-a^{-\frac{1}{m-1}} \sigma^\frac{m}{m-1} - a^{-\frac{1}{m-1}}\\
(\text{\Romannum{3}}): \ &-G(z) \geq -G(z-1) \Leftrightarrow 2-a^\frac{m}{m-1} -b^\frac{m}{m-1} \geq 2b-b^{-\frac{1}{m-1}} \nu^\frac{m}{m-1} - b^{-\frac{1}{m-1}}.
\end{align*}
(\Romannum{2}) and (\Romannum{3}) are equivalent to
\begin{align*}
(\text{\Romannum{2}})^\prime: \ &\sigma^\frac{m}{m-1} \geq 2a^\frac{m}{m-1}-1-2a^\frac{1}{m-1}+a^\frac{m+1}{m-1}+a^\frac{1}{m-1} b^\frac{m}{m-1}\\
(\text{\Romannum{3}})^\prime: \ &\nu^\frac{m}{m-1} \geq 2b^\frac{m}{m-1}-1-2b^\frac{1}{m-1}+b^\frac{m+1}{m-1}+a^\frac{m}{m-1} b^\frac{1}{m-1}.
\end{align*}

We next have a closer look at $\mathcal{D}_{m,1} (u)(z)$. We have
\begin{align*}
\mathcal{D}_{m,1}& (u)(z) = -2 \frac{(m-1)^3}{m^2} v^\frac{m-2}{m-1} (z) Lv^\frac{m}{m-1} (z) - \frac{(m-1)^2}{m^2} v^{-\frac{2}{m-1}} (z) Lv^\frac{m}{m-1} (z) v^\frac{m}{m-1} (z+1) \\ &+ \frac{(m-1)^2}{m} v^{-\frac{1}{m-1}} (z) v (z+1) Lv^\frac{m}{m-1} (z+1) - \frac{(m-1)^2}{m^2} v^{-\frac{2}{m-1}} (z) Lv^\frac{m}{m-1} (z) v^\frac{m}{m-1} (z-1) \\&+ \frac{(m-1)^2}{m} v^{-\frac{1}{m-1}} (z) v (z-1) Lv^\frac{m}{m-1} (z-1).
\end{align*}
Hence,
\begin{align}
\frac{\mathcal{D}_{m,1} (u)(z)}{\frac{(m-1)^2}{m^2} v^2(z)} &= -2 (m-1) \big(a^\frac{m}{m-1}+b^\frac{m}{m-1}-2\big) - a^\frac{m}{m-1} \big(a^\frac{m}{m-1}+b^\frac{m}{m-1}-2\big) \notag \\ 
&\mkern-18mu \mkern-18mu+ m  a \big(\sigma^\frac{m}{m-1} +1-2a^\frac{m}{m-1}\big) - b^\frac{m}{m-1} \big(a^\frac{m}{m-1}+b^\frac{m}{m-1}-2\big) + m b  \big(\nu^\frac{m}{m-1} +1-2b^\frac{m}{m-1}\big).
\label{GleichungDtG}
\end{align}
Now,
\begin{align*}
\frac{\big(-G(z)\big)^2}{\frac{(m-1)^2}{m^2} v^2(z)} &= \big(2-a^\frac{m}{m-1} -b^\frac{m}{m-1}\big)^2 = 4+a^\frac{2m}{m-1}+b^\frac{2m}{m-1} -4a^\frac{m}{m-1}-4b^\frac{m}{m-1} + 2a^\frac{m}{m-1}b^\frac{m}{m-1}.
\end{align*}
Using (\Romannum{2})$^\prime$ and (\Romannum{3})$^\prime$ in (\ref{GleichungDtG}) then yields
\begin{align*}
\frac{\mathcal{D}_{m,1} (u)(z)}{\frac{(m-1)^2}{m^2} v^2(z)} &\geq -2 (m-1) \big(a^\frac{m}{m-1}+b^\frac{m}{m-1}-2\big) - a^\frac{m}{m-1} \big(a^\frac{m}{m-1}+b^\frac{m}{m-1}-2\big)+ m  a \big(1-2a^\frac{m}{m-1}\big) \notag \\ 
&  +ma\big(2a^\frac{m}{m-1}-1-2a^\frac{1}{m-1}+a^\frac{m+1}{m-1}+a^\frac{1}{m-1} b^\frac{m}{m-1}\big) - b^\frac{m}{m-1} \big(a^\frac{m}{m-1}+b^\frac{m}{m-1}-2\big) \\ &+ mb \big(2b^\frac{m}{m-1}-1-2b^\frac{1}{m-1}+b^\frac{m+1}{m-1}+a^\frac{m}{m-1} b^\frac{1}{m-1}\big) + m b  \big(1-2b^\frac{m}{m-1}\big) \\ &\mkern-18mu= (m-1) \big(4+a^\frac{2m}{m-1}+b^\frac{2m}{m-1} -4a^\frac{m}{m-1}-4b^\frac{m}{m-1} + 2a^\frac{m}{m-1}b^\frac{m}{m-1}\big) \\ &\mkern-18mu= (m-1) \frac{\big(-G(z)\big)^2}{\frac{(m-1)^2}{m^2} v^2(z)},
\end{align*}
which yields the statement.
\end{proof}

As a byproduct, Theorem \ref{SatzAltCD} also yields new findings on the square.

\begin{korollar}
\label{KorollarSquare}
For all $m>1$, the operator $L$ from Subsection \ref{SectionSquare} satisfies $CD_{m,1} \big(0,\frac{1}{m-1}\big)$. 
\end{korollar}

\begin{proof}
The proof is identical to the one of Theorem \ref{SatzAltCD} with the only difference that $\sigma = \nu$.
\end{proof}
\section{Aronson-Bénilan estimate}
In this chapter we will prove a discrete version of the Aronson-B\'{e}nilan estimate on finite graphs. As we will see, the condition $CD_{m,\alpha} (0,d)$ will play a key role.
\begin{satz}
\label{ThmAronsonBen}
Let $m >1$, $X\neq \emptyset$ be a finite set, $k: X \times X \to [0,\infty)$, and assume that the operator $L$ generated by $k$ satisfies $CD_{m,\alpha} (0,d)$ with some $d>0$ and $\alpha\in [0,1]$. Suppose that $u \colon [0,\infty) \times X \to (0,\infty)$ is $C^1$ in time and solves the porous medium equation $\partial_t u - Lu^m = 0$ on $(0,\infty) \times X$. Then the function $v = \frac{m}{m-1} u^{m-1}$ satisfies
\begin{align}
-\Big(Lv + \alpha \frac{\tilde{\Psi}_\Upsilon^{(m)} (v)}{(m-1) v}\Big) \leq \frac{d}{t}\;\; \text{ in } (0,\infty) \times X,
\label{LiYauInequ1}
\end{align}
and thus
\begin{align}
(1-\alpha) \frac{\tilde{\Psi}_\Upsilon^{(m)} (v)}{(m-1)v} - \frac{\partial_t v }{(m-1) v} \leq \frac{d}{t}\;\; \text{ in } (0,\infty) \times X.
\label{LiYauInequ2}
\end{align}
\end{satz}
\begin{proof}
On $[0,\infty) \times X$ we define the function $G (t,x) = Lv(t,x) + \alpha \frac{\tilde{\Psi}_\Upsilon^{(m)} (v)}{(m-1) v} (t,x)$ and set
\begin{align*}
\tilde{G} (t,x) = -\frac{t}{d} G(t,x).
\end{align*}
Let $t_1 >0$ be arbitrarily fixed. Suppose that $\tilde{G}$ (restricted to the set $[0,t_1] \times X$) assumes its global maximum at $(t_*,x_*) \in [0,t_1] \times X$. W.l.o.g. we may assume that $\tilde{G} (t_*,x_*)>0$, since otherwise (\ref{LiYauInequ1}) follows directly. By definition of $\tilde{G}$ it is then clear that $t_*>0$ and thus $(\partial_t \tilde{G}) (t_*,x_*) \geq 0$.
As $u$ solves the discrete PME we have 
\begin{align*}
\partial_t v = m u^{m-2} \partial_t u = m u^{m-2} Lu^m 
\end{align*}
and thus 
\begin{align*}
\partial_t &G (t,x) = \partial_t \Big(Lv(t,x) + \alpha \frac{\tilde{\Psi}_\Upsilon^{(m)} (v)}{(m-1) v} (t,x)\Big) \\&= \partial_t \Big(\sum_{y \in X} k(x,y)\Big[(1-\alpha) v(t,y)-\big(1-\frac{\alpha}{m}\big) v(t,x)+ \alpha \frac{m-1}{m} v^{-\frac{1}{m-1}}(t,x) v^\frac{m}{m-1} (t,y) \Big]\Big) \\ 
&= \sum_{y \in X} k(x,y) \Big(\Big[1-\alpha+\alpha \frac{u(y)}{u(x)}\Big] mu^{m-2}(t,y) Lu^m(t,y)\\
& \quad\quad\quad -\Big[m-\alpha+\alpha \Big(\frac{u(t,y)}{u(t,x)}\Big)^m\Big] u^{m-2}(t,x)Lu^m(t,x)\Big) \\
&=\mathcal{D}_{m,\alpha} (u)(t,x) \text{ in } (0,\infty) \times X,
\end{align*}
from which we deduce that
\begin{align*}
\partial_t \tilde{G} (t,x) = -\frac{1}{d} G(t,x) -\frac{t}{d} \mathcal{D}_{m,\alpha} (u) (t,x).
\end{align*}
It follows that at the maximum point $(t_*,x_*)$ we have that
\begin{align*}
0 \leq -G-t_* \mathcal{D}_{m,\alpha} (u) ,
\end{align*}
which is equivalent to 
\begin{align*}
\mathcal{D}_{m,\alpha} (u) \leq \frac{-G}{t_*}
\end{align*}
at $(t_*,x_*)$. Since $-G(t_*,x_*)$ is the global maximum of $-G(t_*,x), \ x \in X,$ we may apply condition $CD_{m,\alpha} (0,d)$, which gives
\begin{align*}
\frac{1}{d} \big(-G(t_*,x_*)\big)^2 \leq \mathcal{D}_m (u)(t_*,x_*) \leq \frac{-G(t_*,x_*)}{t_*},
\end{align*}
and thus
\begin{align*}
\tilde{G} (t_*,x_*) \leq 1.
\end{align*}
Since $(t_*,x_*)$  was a global maximum point of $\tilde{G}$ restricted to the set $[0,t_1] \times X$ with $t_1>0$ arbitrarily chosen, we obtain
\begin{align*}
\tilde{G} (t_1,x) \leq \tilde{G} (t_*,x_*) \leq 1,\quad t_1 \in (0, \infty), \ x \in X.
\end{align*}
This shows (\ref{LiYauInequ1}), which in turn, together with (\ref{DiscreteChainRuleV}), implies inequality (\ref{LiYauInequ2}).
\end{proof}
In the important special case $\alpha=0$, Theorem \ref{ThmAronsonBen} takes the following form.
\begin{korollar}
\label{KorAronsonBen}
Let $m >1$, $X\neq \emptyset$ be a finite set, $k: X \times X \to [0,\infty)$, and assume that the operator $L$ generated by $k$ satisfies $CD_{m} (0,d)$ with some $d>0$. Suppose that $u \colon [0,\infty) \times X \to (0,\infty)$ is $C^1$ in time and solves the equation $\partial_t u - Lu^m = 0$ on $(0,\infty) \times X$. Then the function $v = \frac{m}{m-1} u^{m-1}$ satisfies
\begin{align*}
-Lv \leq \frac{d}{t}\quad \text{ in } (0,\infty) \times X,
\end{align*}
and thus
\begin{align*}
\frac{\tilde{\Psi}_\Upsilon^{(m)} (v)}{(m-1)v} - \frac{\partial_t v }{(m-1) v} \leq \frac{d}{t}\quad \text{ in } (0,\infty) \times X.
\end{align*}
\end{korollar}
\begin{proof}
Set $\alpha=0$ in Theorem \ref{ThmAronsonBen}.
\end{proof}

\begin{beispiel}
\label{BspLiYau}
(i) \emph{The square.} Let $X$ and $L$ be as in Section \ref{SectionSquare}. By Corollary \ref{KorollarSquare}, we know that for any $m>1$ the operator $L$ satisfies $CD_{m,1} \big(0,\frac{1}{m-1}\big)$. With Theorem \ref{ThmAronsonBen} we find that for any $u \colon [0,\infty) \times X \to (0,\infty)$ which solves the equation $\partial_t u - Lu^m =0$ on $(0,\infty) \times X$, the function $v= \frac{m}{m-1} u^{m-1}$ enjoys the estimate
\begin{align*}
-\Big(Lv +  \frac{\tilde{\Psi}_\Upsilon^{(m)} (v)}{(m-1) v}\Big) \leq \frac{1}{(m-1) t}\quad \text{ in } (0,\infty) \times X,
\end{align*} 
or equivalently
\begin{align*}
- \frac{\partial_t v }{v} \leq \frac{1}{(m-1) t}\quad \text{ in } (0,\infty) \times X.
\end{align*}
In the special case $m=2$, we further know from Section \ref{SectionSquare} that $L$ satisfies $CD_2 \big(0,\frac{4}{3}\big)$. For $v$ as above we then get with Corollary \ref{KorAronsonBen}
\begin{align*}
-Lv (t,x) \leq \frac{4}{3t}\quad \text{ in } (0,\infty) \times X,
\end{align*} 
or equivalently
\begin{align*}
\frac{\tilde{\Psi}_\Upsilon^{(2)} (v)}{v} - \frac{\partial_t v }{v} \leq \frac{4}{3t}\quad \text{ in } (0,\infty) \times X.
\end{align*}
(ii) \emph{Complete graphs.} Consider again the complete, unweighted graph with $D$ vertices ($D \in \N$) and let the operator $L$ be as in Example \ref{BspVollstGraph}. We already saw that $L$ satisfies $CD_m (0,d)$ with $d= \frac{m}{(m-1)^2}$, which is optimal for $m>2$. In this case, Corollary \ref{KorAronsonBen} shows that any positive solution $u$ to the porous medium equation $\partial_t u - L u^m = 0$ on $(0,\infty) \times X$ satisfies the Aronson-Bénilan estimate
\begin{align*}
-Lv (t,x) \leq \frac{m}{(m-1)^2 t}\quad \text{ in } (0,\infty) \times X,
\end{align*}
where $v= \frac{m}{m-1} u^{m-1}$. As we saw, this estimate can be sharpened in the case $m \in (1,2]$. For example, if $m=2$ there holds
\begin{align}
-Lv (t,x) \leq \frac{2(D-1)}{D t}\quad \text{ in } (0,\infty) \times X.
\label{BspD=2}
\end{align}

In the case $D=2$ and $m=2$ we can check the quality of our estimate by calculating the exact solution. To this end, consider a positive solution $u(t,x) = \big(u_1 (t),u_2(t) \big)$ to the PME $\partial_t u -Lu^2 = 0$ on 
$(0,\infty) \times  \{x_1,x_2\}$, i.e. $u$ satisfies
\begin{align*}
\begin{cases}
\dot{u}_1 = u_2^2-u_1^2, \ u_1(0) =: a_1 >0, \\
\dot{u}_2 = u_1^2-u_2^2, \ u_2(0) =: a_2 >0.
\end{cases}
\end{align*}
Since $\dot{u}_1 + \dot{u}_2 = 0$ for all $t>0$ we have $u_1 (t) + u_2(t) = a_1+a_2 =: \lambda$ for all $t\geq 0$
(conservation of mass). Inserting $u_2 = \lambda-u_1$ into the first equation yields the linear equation
\begin{align*}
\dot{u}_1 = (\lambda - u_1)^2 -u_1^2 = \lambda^2 - 2\lambda u_1,
\end{align*}
which can be solved by variation of constants. This gives
\begin{align*}
u_1(t) = e^{-2\lambda t} a_1 + \lambda^2 \int_0^t e^{-2 \lambda (t-s)} ds = \frac{a_1-a_2}{2} e^{-2\lambda t} + \frac{\lambda}{2},
\end{align*}
and thus 
\begin{align*}
u_2(t) = \frac{a_2-a_1}{2} e^{-2\lambda t} + \frac{\lambda}{2}.
\end{align*}
For the function $v = \frac{m}{m-1} u^{m-1} = 2 u$, the previous relations yield
\begin{align*}
-Lv(t,x_1) = -2Lu_1(t) = 2 u_1(t) - 2u_2(t) = 2 (a_1-a_2) e^{-2(a_1+a_2) t}
\end{align*}
and 
\begin{align*}
-Lv(t,x_2) = -2Lu_2(t) = 2 u_2(t) - 2u_1(t) = 2 (a_2-a_1) e^{-2(a_1+a_2) t}.
\end{align*}
W.l.o.g. we may assume $a_1 \geq a_2$. Then
\begin{align*}
-Lv(t,x) \leq -Lv(t,x_1) \leq 2a_1 e^{-2a_1 t},
\end{align*}
which corresponds to the limit $a_2 \to 0$. Now, maximising the right-hand side (i.e.\ the function $f(x) = x e^{-tx}$ over $[0,\infty)$) yields
\begin{align*}
-Lv(t,x) \leq \frac{1}{et}.
\end{align*}
Comparing (\ref{BspD=2}) for $D=2$ with this optimal result, we see that we only miss the sharp constant by a factor of $e$.
\end{beispiel}

\section{Harnack inequality}
The aim of this section is to derive a Harnack inequality for solutions to the discrete PME by means of the Aronson-Benilan estimate, which can be viewed as a differential Harnack estimate. We will need the following technical lemma on the function $\tilde{\Upsilon}$ from Definition \ref{DefTildePsi}.
\begin{lemma}
\label{LemmaHarnack}
(i) If $m\in (1,2]$, then for all $x \in [1,\infty)$ there holds
\begin{align}
\tilde{\Upsilon} (\log x) \geq \frac{1}{2} (x-1)^2.
\label{UngleichungTildeUps}
\end{align}

(ii) If $m\in [2,\infty)$, then for all $x \in (0,1]$ we have
\begin{align}
\tilde{\Upsilon} (\log x) \geq \frac{1}{2} (x-1)^2.
\label{UngleichungTildeUps2}
\end{align}
\end{lemma}
\begin{proof}
We consider the function 
\begin{align*} q(x) : &= \tilde{\Upsilon} (\log x) -\frac{1}{2} (x-1)^2\\
& = \frac{(m-1)^2}{m} x^\frac{m}{m-1} -(m-1)x+\frac{m-1}{m}-\frac{1}{2} (x-1)^2, \ x > 0.
\end{align*}
Clearly $q(1)=0$. Thus, for (i) and (ii) it is sufficient to show that $q^\prime \leq 0$ in $(0,1]$ if $m\ge 2$ and that $q^\prime \geq 0$ in $[1,\infty)$ for $m\in (1,2]$. Now, 
\begin{align*}
q^\prime (x) = (m-1) x^\frac{1}{m-1} -x -(m-2), \ x > 0.
\end{align*}
Setting $h(x) = (m-1) x^\frac{1}{m-1} -(m-1), \ x > 0$, we have 
\begin{align*}
h^\prime (x) = x^\frac{2-m}{m-1} \quad \text{and}\quad h^{\prime \prime} (x) = \frac{2-m}{m-1} x^\frac{3-2m}{m-1}.
\end{align*}
Consequently, $h^{\prime \prime} \geq 0 \ (h^{\prime \prime} \le 0)$ in $(0,\infty)$ for $m \in (1,2] \ (m\ge  2)$ and hence $h$ is convex (concave) on $(0,\infty)$ for $m \in (1,2] \ (m\ge  2)$. Using this, together with $h(1)=0$ and $h^\prime(1)=1$, we see
that
\begin{align*}
q^\prime (x) = h(x)-h(1)-h^\prime(1)(x-1) \geq (\leq) \; 0,\;x>0,
\end{align*}
for $m \in (1,2] \ (m\ge  2)$, which concludes the proof.
\end{proof}

\begin{bemerkung}
The factor $\frac{1}{2}$ in (\ref{UngleichungTildeUps}) and (\ref{UngleichungTildeUps2}) is optimal as $\frac{\tilde{\Upsilon} (\log x)}{(x-1)^2} \to \frac{1}{2} \ (x \to 1)$ for any $m >1$.
\end{bemerkung}



We now come to the main theorem of this section. For a connected finite graph with vertex set $X$ and symmetric edge
weights we define the distance $d(x_1,x_2)$ of two points $x_1, x_2\in X$ as the minimal
length of all paths connecting $x_1$ and $x_2$ within the given graph. Here by the length of a path we mean the total number
of involved edges. 
\begin{satz}
\label{SatzHarnack}
Let $m >1$ and $X\neq \emptyset$ be a finite set. Let $k: X \times X \to [0,\infty)$ be symmetric and such that the induced graph with vertex set $X$ is connected. Let further $L$ be the Laplace operator generated by $k$. 
Suppose that $u \colon [0,\infty) \times X \to (0,\infty)$ is $C^1$ in time and that the function $v=\frac{m}{m-1} u^{m-1}$  satisfies the differential Harnack estimate 
\begin{equation}
\partial_t v \geq (1-\lambda) \tilde{\Psi}_\Upsilon^{(m)} (v) -  \frac{\mu}{t}v\quad \text{ in } (0,\infty) \times X
\label{AssHarnack}
\end{equation}
with some constants $\lambda \in [0,1)$ and $\mu>0$.  Let $0<t_1<t_2$ and $x_1,x_2 \in X$. Then for any sequence of pairwise distinct points $(y_i)_{i=0,1,\dots,N}, \ N \in \mathbb{N},$ such that $y_0 = x_1$, $y_N = x_2$ and $k(y_{i-1},y_{i}) >0, \ i \in \{1,\dots,N\},$ there holds
\begin{equation}
t_1^\mu v(t_1,x_1) \leq t_2^\mu v(t_2,x_2) + \frac{2 N^2}{(1-\lambda)(\mu+1)\left(t_2-t_1 \right)^2} \sum_{j=1}^N \frac{\tau_j^{\mu+1}-\tau_{j-1}^{\mu+1}}{k(y_{j-1},y_j)},
\label{HarnackUngl}
\end{equation} 
where $\tau_i = t_1 + i \frac{t_2-t_1}{N}, \ i =0,1,\dots,N$. In particular,
\begin{equation}
t_1^\mu v(t_1,x_1) \leq t_2^\mu v(t_2,x_2) + \frac{2 d(x_1,x_2)^2 \left(t_2^{\mu+1}-t_1^{\mu+1} \right)}
{(1-\lambda)(\mu+1) k_{\min} \left(t_2-t_1 \right)^2},
\label{HarnackUngl2}
\end{equation} 
where $k_{\min}$ is the minimal (positive) weight of an edge in $X \times X$.
\end{satz}
The following lemma will be a central tool in the proof. It can be seen as an analogue to \cite[Lemma 5.3]{BHL}.
\begin{lemma}
\label{LemmaHarnack2}
Let $0<t_1<t_2$ and $c,\nu>0$. Then for any continuous function $\psi: [t_1,t_2] \to \R$ we have
\begin{align}
\min_{s \in [t_1,t_2]} &\Big(\psi (s) -\frac{1}{c} \int_{s}^{t_2} \tau^{-\nu} \psi^2 (\tau) d\tau\Big) \leq
\frac{c}{\nu+1} \frac{t_2^{\nu+1}-t_1^{\nu+1}}{ (t_2-t_1)^2}
\label{LemmaHarnackUGL1}
\end{align}
and
\begin{align}
\min_{s \in [t_1,t_2]} &\Big(\psi (s) -\frac{1}{c} \int_{t_1}^{s} \tau^{-\nu} \psi^2 (\tau) d\tau\Big) 
\leq \frac{c}{\nu+1}  \frac{t_2^{\nu+1}-t_1^{\nu+1}}{(t_2-t_1)^2}.
\label{LemmaHarnackUGL2}
\end{align}
\end{lemma}
\begin{proof}
Let the functions $\eta_1, \eta_2$ be defined by $\eta_1 (t) := \frac{2}{c} (t-t_1)$ and $\eta_2 (t) := \frac{2}{c} (t_2-t)$. Then
\begin{align*}
\min_{s \in [t_1,t_2]} &\Big(\psi (s) - \frac{1}{c} \int_s^{t_2} \tau^{-\nu} \psi^2 (\tau) d\tau\Big) \leq \frac{\int_{t_1}^{t_2} \Big(\eta_1 (t) \psi (t) - \frac{1}{c}  \int_t^{t_2} \eta_1 (t) \tau^{-\nu} \psi^2 (\tau) d\tau\Big) dt}{\int_{t_1}^{t_2} \eta_1 (t) dt} \\ &= \frac{c}{(t_2-t_1)^2} \int_{t_1}^{t_2} \Big(\eta_1 (t) \psi (t) - \frac{1}{c} t^{-\nu} \psi^2 (t) \int_{t_1}^{t} \eta_1 (\tau) d\tau\Big) dt \\ &= \frac{c}{(t_2-t_1)^2} \int_{t_1}^{t_2} \Big(\eta_1 (t) \psi (t) - \frac{1}{4} t^{-\nu} \big(\eta_1 (t) \psi(t)\big)^2 \Big) dt. 
\end{align*}
Similarly,
\begin{align*}
\min_{s \in [t_1,t_2]} &\Big(\psi (s) - \frac{1}{c} \int_{t_1}^{s} \tau^{-\nu} \psi^2 (\tau) d\tau\Big) \leq \frac{\int_{t_1}^{t_2} \Big(\eta_2 (t) \psi (t) - \frac{1}{c}  \int_{t_1}^{t} \eta_2 (t) \tau^{-\nu} \psi^2 (\tau) d\tau\Big) dt}{\int_{t_1}^{t_2} \eta_2 (t) dt} \\ &= \frac{c}{(t_2-t_1)^2} \int_{t_1}^{t_2} \Big(\eta_2 (t) \psi (t) - \frac{1}{c} t^{-\nu} \psi^2 (t) \int_{t}^{t_2} \eta_2 (\tau) d\tau\Big) dt \\ &= \frac{c}{(t_2-t_1)^2} \int_{t_1}^{t_2} \Big(\eta_2 (t) \psi (t) - \frac{1}{4} t^{-\nu} \big(\eta_2 (t) \psi(t)\big)^2 \Big) dt. 
\end{align*}
Now, analysing the function $f(x) = x- \frac{1}{4} t^{-\nu} x^2$, one can show that 
\begin{align*}
f(x) \leq t^\nu, \ x \in \R.
\end{align*} 
Thus, 
\begin{align*}
\min_{s \in [t_1,t_2]} \Big(\psi (s) - \frac{1}{c} \int_s^{t_2} \tau^{-\nu} \psi^2 (\tau) d\tau\Big) &\leq \frac{c}{ (t_2-t_1)^2} \int_{t_1}^{t_2} t^\nu dt = \frac{c}{\nu+1} \frac{t_2^{\nu+1}-t_1^{\nu+1}}{(t_2-t_1)^2}
\end{align*}
and
\begin{align*}
\min_{s \in [t_1,t_2]}\Big( \psi (s) - \frac{1}{c} \int_{t_1}^{s} \tau^{-\nu} \psi^2 (\tau) d\tau\Big) &\leq \frac{c}{ (t_2-t_1)^2} \int_{t_1}^{t_2} t^\nu dt = \frac{c}{\nu+1} \frac{t_2^{\nu+1}-t_1^{\nu+1}}{(t_2-t_1)^2}.
\end{align*}
\end{proof}



\begin{proof}[Proof of Theorem \ref{SatzHarnack}]
Multiplying (\ref{AssHarnack}) by $t^\mu$, we find
\begin{align*}
t^\mu \partial_t v +\mu t^{\mu-1} v \geq (1-\lambda) t^\mu \tilde{\Psi}_\Upsilon^{(m)} (v),
\end{align*}
which is equivalent to 
\begin{align}
\partial_t \big(t^\mu v \big) \geq (1-\lambda) t^\mu \tilde{\Psi}_\Upsilon^{(m)} (v).
\label{BewHarnackDtv}
\end{align}

We first consider the case $N=1$. Let $0 < t_1 < t_2$ and $s \in J := [t_1, t_2]$. Then we may write
\begin{align*}
t_1^\mu v(&t_1,x_1) - t_2^\mu v(t_2,x_2) \\ &= t_1^\mu v(t_1,x_1) -s^\mu v(s,x_1) +s^\mu v(s,x_1) -s^\mu v(s,x_2) +s^\mu v(s,x_2) -t_2^\mu v(t_2,x_2) \\&= - \int_{t_1}^s \partial_t \big( t^\mu v (t,x_1)\big) dt + s^\mu v(s,x_1) -s^\mu v(s,x_2) - \int_s^{t_2} \partial_t \big(t^\mu v (t,x_2)\big) dt =: (*).
\end{align*}
Define $\delta (t) := t^\mu \big(v(t,x_1)-v(t,x_2)\big), \ t \in J$. We distinguish two cases. Let first $m \in (1,2]$. Then by (\ref{BewHarnackDtv})
\begin{align*}
(*) &\leq \delta (s) - (1-\lambda) \int_s^{t_2} t^\mu \tilde{\Psi}_\Upsilon^{(m)} (v) (t,x_2) dt \\ &\leq \delta (s) - (1-\lambda) k(x_1,x_2) \int_s^{t_2} t^\mu v^2 (t,x_2) \tilde{\Upsilon} \Big( \log \frac{v(t,x_1)}{v(t,x_2)}\Big) dt.
\end{align*}
We choose $s \in J$ such that the continuous function $\omega_1$ defined by
\begin{align*}
\omega_1 (t) := \delta (t) -(1-\lambda) k(x_1,x_2) \int_t^{t_2} \tau^\mu v^2 (\tau,x_2) \tilde{\Upsilon} \Big( \log \frac{v(\tau,x_1)}{v(\tau,x_2)}\Big) d\tau
\end{align*}
attains its minimum at $s$. Suppose that $\omega_1 (s) \geq 0$. Then the positivity of $\tilde{\Upsilon}$ implies that $\delta (t) \geq 0, \ t \in J$. Thus, by Lemma \ref{LemmaHarnack}, it follows that
\begin{align*}
\omega_1 (t) &\leq \delta (t) - \frac{(1-\lambda) k(x_1,x_2)}{2} \int_t^{t_2} \tau^\mu \big(v(\tau,x_1)-v(\tau,x_2)\big)^2 d\tau \\&= \delta (t) - \frac{(1-\lambda) k(x_1,x_2)}{2} \int_t^{t_2} \tau^{-\mu} \delta(\tau)^2 d\tau =: \tilde{\omega}_1 (t).
\end{align*}
By (\ref{LemmaHarnackUGL1}), we then have
\begin{align*}
\min_{t \in J} \tilde{\omega}_1 (t) \leq \frac{2}{(1-\lambda) k(x_1,x_2)(\mu+1)} \frac{t_2^{\mu+1}-t_1^{\mu+1}}{(t_2-t_1)^2}
\end{align*}
and as $\omega_1 (s) \leq \min_{t \in J} \tilde{\omega}_1 (t)$ we infer that
\begin{align*}
\omega_1 (s) \leq \frac{2}{(1-\lambda) k(x_1,x_2)(\mu+1)} \frac{t_2^{\mu+1}-t_1^{\mu+1}}{(t_2-t_1)^2}.
\end{align*}
Note that this estimate is trivial when $\omega_1 (s) <0$. With the estimate above, this finally yields
\begin{align*}
t_1^\mu v(&t_1,x_1) - t_2^\mu v(t_2,x_2) \leq \frac{2}{(1-\lambda) k(x_1,x_2)(\mu+1)} \frac{t_2^{\mu+1}-t_1^{\mu+1}}{(t_2-t_1)^2}.
\end{align*}

If $m \in (2,\infty)$ we estimate
\begin{align*}
(*) &\leq \delta (s) - (1-\lambda) \int_{t_1}^{s} t^\mu \tilde{\Psi}_\Upsilon^{(m)} (v) (t,x_1) dt \\ &\leq \delta (s) - (1-\lambda) k(x_1,x_2) \int_{t_1}^{s} t^\mu v^2 (t,x_1) \tilde{\Upsilon} \Big( \log \frac{v(t,x_2)}{v(t,x_1)}\Big) dt.
\end{align*}
In this case, we choose $s \in J$ such that the continuous function $\omega_2$ defined by
\begin{align*}
\omega_2(t) := \delta (t) -(1-\lambda) k(x_1,x_2) \int_{t_1}^{t} \tau^\mu v^2 (\tau,x_1) \tilde{\Upsilon} \Big( \log \frac{v(\tau,x_2)}{v(\tau,x_1)}\Big) d\tau
\end{align*}
assumes its minimum at $s$. Supposing that $\omega_2 (s) \geq 0$ we have $\delta (t) \geq 0, \ t \in J$, by positivity of $\tilde{\Upsilon}$. Consequently, the second part of Lemma \ref{LemmaHarnack} gives
\begin{align*}
\omega_2 (t) &\leq \delta (t) - \frac{(1-\lambda) k(x_1,x_2)}{2} \int_{t_1}^{t} \tau^\mu \big(v(\tau,x_1)-v(\tau,x_2)\big)^2 d\tau \\&= \delta (t) - \frac{(1-\lambda) k(x_1,x_2)}{2} \int_{t_1}^{t} \tau^{-\mu} \delta(\tau)^2 d\tau =: \tilde{\omega}_2 (t).
\end{align*}
By (\ref{LemmaHarnackUGL2}), we then have
\begin{align*}
\min_{t \in J} \tilde{\omega}_2 (t) \leq \frac{2}{(1-\lambda) k(x_1,x_2)(\mu+1)} \frac{t_2^{\mu+1}-t_1^{\mu+1}}{(t_2-t_1)^2}.
\end{align*}
and as $\omega_2 (s) \leq \min_{t \in J} \tilde{\omega}_2 (t)$, it follows that
\begin{align*}
\omega_2 (s) \leq \frac{2}{(1-\lambda) k(x_1,x_2)(\mu+1)} \frac{t_2^{\mu+1}-t_1^{\mu+1}}{(t_2-t_1)^2},
\end{align*}
which clearly also holds when $\omega_2 (s) <0$. Again, this yields
\begin{align*}
t_1^\mu v(&t_1,x_1) - t_2^\mu v(t_2,x_2) \leq \frac{2}{(1-\lambda) k(x_1,x_2)(\mu+1)} \frac{t_2^{\mu+1}-t_1^{\mu+1}}{(t_2-t_1)^2}.
\end{align*}
This shows the statement for $N=1$. 

For the case $N>1$ consider a sequence of pairwise distinct points $(y_i)_{\{i=0,1,\dots,N\}}$ such that $y_0 = x_1$, $y_N = x_2$ and $k(y_{l-1},y_{l}) >0, \ l \in \{1,\dots,N\}$. Defining the times $\tau_i = t_1 + i \frac{t_2-t_1}{N}, \ i =0,1,\dots,N,$ and employing the result for $N=1$, we obtain
\begin{align*}
t_1^\mu &v(t_1,x_1) - t_2^\mu v(t_2,x_2) = \sum_{j=1}^N \big( \tau_{j-1}^\mu v(\tau_{j-1},y_{j-1}) - \tau_{j}^\mu v(\tau_{j},y_{j})\big) \\&\leq \frac{2}{(1-\lambda) (\mu+1)} \sum_{j=1}^N \frac{\tau_j^{\mu+1}-\tau_{j-1}^{\mu+1}}{k(y_{j-1},y_{j}) (\tau_j-\tau_{j-1})^2} \\ &= \frac{2 N^2}{(1-\lambda) (\mu+1)(t_2-t_1)^2} \sum_{j=1}^N \frac{\tau_j^{\mu+1}-\tau_{j-1}^{\mu+1}}{k(y_{j-1},y_j)}.
\end{align*}
The inequality (\ref{HarnackUngl2}) now follows easily, as
\begin{align*}
\sum_{j=1}^N \frac{\tau_j^{\mu+1}-\tau_{j-1}^{\mu+1}}{k(y_{j-1},y_j)} \leq \frac{1}{k_{\min}} \sum_{j=1}^N \left( \tau_j^{\mu+1}-\tau_{j-1}^{\mu+1} \right) = \frac{t_2^{\mu+1}-t_1^{\mu+1}}{k_{\min}}.
\end{align*}

Finally, observe that the assertion \eqref{HarnackUngl2} is also valid in the case $x_1=x_2$.
\end{proof}
\begin{korollar}
\label{KorollarHarnack}
Let $m>1$ and $L$ be as in Theorem \ref{SatzHarnack}. Suppose that $L$ satisfies $CD_{m,\alpha} (0,d)$ with some $d>0$ and $\alpha \in [0,1)$. If $u \colon [0,\infty) \times X \to (0,\infty)$ is a solution to the PME $\partial_t u - Lu^m = 0$ on $(0,\infty) \times X$, then for $v=\frac{m}{m-1} u^{m-1}$ the estimates (\ref{HarnackUngl}) and (\ref{HarnackUngl2}) hold 
true with $\mu = (m-1) d$ and $\lambda = \alpha$.
\end{korollar}
\begin{proof}
By Theorem \ref{ThmAronsonBen}, $L$ satisfies (\ref{AssHarnack}) with $\mu = (m-1)d$ and $\lambda = \alpha$.
\end{proof}

\begin{bemerkung} 
\emph{(Comparison with the heat equation)} The Harnack inequalities (\ref{HarnackUngl}) and (\ref{HarnackUngl2}) are closely related to the ones from \cite{DKZ} and \cite{KWZ}, where the authors considered solutions to the discrete heat equation, i.e. equation (\ref{PME}) with $m=1$. Suppose that $L$ satisfies $CD_{m} (0,d)$ with some $C^1$-function $d = d(m)>0$, $m\ge 1$. As we saw in Remark \ref{BemerkungCDBed}, $CD_1 (0,d(1))$ is equivalent to the condition $CD(F;0)$ from \cite{DKZ} with a quadratic $CD$-function $F$. In \cite{DKZ} it has been shown that under this condition the Harnack estimate 
\begin{align}
\label{HarnackDKZ}
u(t_1,x_1) \leq u(t_2,x_2) \left( \frac{t_2}{t_1} \right) ^{d(1)} \exp\left(\frac{2 d(x_1,x_2)^2}{k_{\min} (t_2-t_1)} \right)
\end{align}
holds for positive solutions $u$ to the heat equation, see \cite[Theorem 6.1]{DKZ}. Here, we adapted the result to our setting and the notation from Theorem \ref{SatzHarnack}. This is exactly what one gets when taking the limit $m \to 1$ in (\ref{HarnackUngl2}).

Indeed, by Corollary \ref{KorollarHarnack} with $\alpha=0$ we have $\mu = \mu(m) =(m-1) d(m)$ satisfying $\mu(1) =0$. Note that $\mu$ is $C^1$ with $\mu'(1)=d(1)$. Now, letting $m \to 1$ in (\ref{HarnackUngl}), we deduce the corresponding statement for the heat equation recently stated in \cite[Theorem 6.1]{KWZ}. Indeed, we have 
\begin{align*}
t_1^\mu v(t_1,x_1) - t_2^\mu v(t_2,x_2) &= \frac{m t_1^{(m-1)d(m)} u^{m-1} (t_1,x_1) -mt_2^{(m-1)d(m)} u^{m-1}(t_2,x_2)}{m-1} \\ &\to d(1) \log \frac{t_1}{t_2} + \log \frac{u(t_1,x_1)}{u(t_2,x_2)} \quad \mbox{as}\;m \to 1,
\end{align*}
by l'H\^{o}spital's rule, and thus as $m \to 1$, (\ref{HarnackUngl}) becomes
\begin{align*}
\log \frac{u(t_1,x_1)}{u(t_2,x_2)} \leq \log \left( \frac{t_2}{t_1} \right)^{d(1)} + \frac{2N}{t_2-t_1} \sum_{j=1}^N \frac{1}{k(y_{j-1},y_j)},
\end{align*}
which is the result from \cite[Theorem 6.1]{KWZ} for a quadratic $CD$-function $F$. Choosing a path with
$N=d(x_1,x_2)$ and as 
\begin{align*}
\sum_{j=1}^N \frac{1}{k(y_{j-1},y_j)} \leq \frac{N}{k_{\min}},
\end{align*}
we finally deduce inequality (\ref{HarnackDKZ}).
\end{bemerkung}

\begin{beispiel}
(i) \emph{The square for m=2.} Let $X$ and $L$ be as in Section \ref{SectionSquare}. From there, we know that $L$ satisfies $CD_2 \big(0, \frac{4}{3}\big)$. Let now $u$ be a positive solution of $\partial_t u - Lu^2 =0$ on $(0,\infty) \times X$ and $v = 2u$. Then, for $0<t_1<t_2$ and $x_1,x_2 \in X$ Corollary \ref{KorollarHarnack} yields
\begin{align*}
t_1^\frac{4}{3} u(t_1,x_1) \leq t_2^\frac{4}{3} u(t_2,x_2) + \frac{7 d(x_1,x_2)^2 \Big(t_2^\frac{7}{3} -t_1^\frac{7}{3}\Big)}{3 (t_2-t_1)^2}.
\end{align*}

(ii) \emph{Complete graphs.} We consider the unweighted complete graph with $D \in \N$ vertices from Example \ref{BspVollstGraph}. In Example \ref{BspLiYau} (ii), we saw that for any $m>1$ the operator $L$ satisfies condition (\ref{AssHarnack}) with $\mu = \frac{m}{m-1}$ and $\lambda=0$ for any positive solution $u$ of $\partial_t u -Lu^m = 0$ on $(0,\infty) \times X$ . We thus find for the pressure function $v$ that
\begin{align*}
t_1^\frac{m}{m-1} v(t_1,x_1) \leq t_2^\frac{m}{m-1} v(t_2,x_2) + \frac{2(m-1)d(x_1,x_2)^2 \Big(t_2^\frac{2m-1}{m-1} -t_1^\frac{2m-1}{m-1}\Big)}{(2m-1)(t_2-t_1)^2}.
\end{align*}
Note that $d(x_1,x_2)=1$ whenever $x_1\neq x_2$.
For $m\leq 2$ this estimate can be improved. For example, when $m=2$, we know that (\ref{AssHarnack}) holds with $\mu = \frac{2(D-1)}{D}$ and $\lambda=0$ for any positive solution to $\partial_t u -Lu^2 = 0$ on $(0,\infty) \times X$. For $0<t_1<t_2$ and $x_1,x_2 \in X$, we then get
\begin{align*}
t_1^\frac{2(D-1)}{D} v(t_1,x_1) \leq t_2^\frac{2(D-1)}{D} v(t_2,x_2) + \frac{2D 
d(x_1,x_2)^2\Big(t_2^\frac{3D-2}{D} -t_1^\frac{3D-2}{D}\Big)}{(3D-2)(t_2-t_1)^2}.
\end{align*}
\end{beispiel}


$\mbox{}$
{\footnotesize

$\mbox{}$


}

\end{document}